\documentclass[a4paper,reqno,11pt,oneside]{amsart}
\usepackage{amsmath,geometry}
\usepackage{graphicx}
\usepackage{float}  
\usepackage{subfig}  
\usepackage{mathrsfs}
\usepackage{amssymb,amsmath, amsfonts, amsthm,color}
\usepackage{latexsym}
\usepackage{color}
\usepackage[dvips]{epsfig}
\usepackage[colorlinks]{hyperref}
\usepackage{comment}
\usepackage{cite}
 
\usepackage{tikz}
\usepackage{pgfplots}
\usetikzlibrary{patterns}

\numberwithin{equation}{section}

\newcommand\G{\mathcal{G}}
\newcommand\E{\mathbb{E}}
\newcommand\V{\mathbb{V}}

\newcommand\vv{\textsc{v}}

\newcommand{\starN}{\mathcal S_N}

\newcommand{\LL}{ \mathscr L}
\newcommand{\NN}{ \mathscr N}
\newcommand{\EE}{ \mathscr E}

\newcommand{\oG}{\overline G}

\usetikzlibrary{shapes.geometric,calc,decorations.markings,math}
\usetikzlibrary{trees,snakes,decorations.pathreplacing}
\tikzstyle{nodo}=[circle,draw,fill,inner sep=0pt,minimum size=0.5*width("k")]
\tikzstyle{infinito}=[circle,inner sep=0pt,minimum size=0mm]
\tikzset{every loop/.style={min distance=10mm,in=300,out=240,looseness=10}}
\tikzset{place/.style={circle,thick,draw=blue!75,fill=blue!20,minimum
		size=6mm}}
\tikzset{place2/.style={circle,thick,draw=red!75,fill=red!20,minimum
		size=6mm}}

\newcommand{\eps}{\varepsilon}

\newcommand{\R}{\mathbb{R}}
\renewcommand{\(}{\left(}
\renewcommand{\)}{\right)}

\newcommand{\dx}{\,dx}

\newtheorem{theorem}{Theorem}[section]
\newtheorem*{theorem*}{Theorem}
\newtheorem{lemma}[theorem]{Lemma}

\newtheorem{proposition}[theorem]{Proposition}
\newtheorem{corollary}[theorem]{Corollary}
\theoremstyle{definition}
\newtheorem{remark}[theorem]{Remark}
\newcommand{\dg}{\text{\normalfont deg}}

\title[ ]{Existence and multiplicity of Peaked bound states for nonlinear Schr\"odinger equations on metric graphs}

\author{Haixia Chen}
\address[H. Chen]{
	Department of mathematics, College of natural sciences, Hanyang University, 222 Wangsimni-ro Seongdong-gu, 04763 Seoul (Republic of Korea)}
\email{hxchen29@hanyang.ac.kr}

\author{Simone Dovetta}
\address[S. Dovetta]{
Dipartimento di Scienze Matematiche, Politecnico di Torino, Corso Duca degli Abruzzi 24 10129 Torino (Italy)
}
\email{simone.dovetta@polito.it}

\author{Angela Pistoia}
\address[A. Pistoia]{Dipartimento SBAI,  Sapienza Università di Roma, Via Antonio Scarpa 16
	00161 Roma (Italy)}
\email{angela.pistoia@uniroma1.it}

\author{Enrico Serra}
\address[E. Serra]{Dipartimento di Scienze Matematiche, Politecnico di Torino, Corso Duca degli Abruzzi 24 10129 Torino (Italy)
}
\email{enrico.serra@polito.it}
\date\today
\thanks{H. Chen is partially supported by the NSFC grants (No.12071169). S. Dovetta is partially supported by the PRIN 2022 project E53D23005450006 and by the INdAM-GNAMPA 2023 project {\em Modelli nonlineari in presenza di interazioni puntuali}. A. Pistoia is partially supported by INdAM-GNAMPA funds and Fondi di Ateneo ``La Sapienza" Universit\`a di Roma (Italy).\\
Data sharing not applicable to this article as no datasets were generated or analysed during the current study.}

\begin{document}
	
	\begin{abstract} We establish existence and multiplicity of one-peaked and multi-peaked positive bound states for nonlinear Schr\"odinger equations on general compact and noncompact metric graphs. Precisely, we construct solutions concentrating at every vertex of odd degree greater than or equal to $3$. We show that these solutions are not minimizers of the associated action and energy functionals. To the best of our knowledge, this is the first work exhibiting solutions concentrating at vertices with degree different than $1$. The proof is based on a suitable Ljapunov-Schmidt reduction.
	\end{abstract}

	\maketitle

	\section{Introduction}
	In this paper we are interested in existence and multiplicity of positive solutions of nonlinear Schr\"odinger equations
	\begin{equation}
	\label{NLS}
	-u''+\lambda u=u^{2\mu+1}\,,
	\end{equation}
	where $\mu>0$ and $\lambda\in\R$, on metric graphs.
	
	The class of graphs we will consider is rather general. In what follows, we assume that $\G=(\V,\E)$ is a connected metric graph such that
	\begin{itemize}
		\item the set of vertices $\V$ and that of edges $\E$ are at most countable;
		
		\item the degree $\dg(\vv)$ of a vertex $\vv$, i.e. the number of edges incident at it, is finite for every $\vv\in\V$;
		
		\item the length $|e|$ of the edge $e$ is bounded away from zero uniformly on $e\in\E$, that is $\displaystyle\inf_{e\in\E}|e|>0$.
	\end{itemize}
	For the sake of brevity, we will say that $\G$ belongs to the class $\mathbf{G}$ whenever it satisfies these conditions. As usual, each bounded edge $e\in\E$ is identified with an interval $[0,\ell_e]$ with $\ell_e:=|e|$, whereas each unbounded edge (if any) is identified with (a copy of) the positive half-line $\R^+$. Graphs in $\mathbf{G}$ are noncompact as soon as they have at least one unbounded edge or  have infinitely many edges. Figure \ref{graph} shows a typical example of metric graph in $\mathbf{G}$. For standard definitions of functional spaces on graphs we redirect e.g. to \cite{BK}.
	
	\begin{figure}[t]
		\begin{center}
			\begin{tikzpicture}
			[scale=1.3,style={circle,
				inner sep=0pt,minimum size=7mm}]
			\node at (0,0) [nodo] (1) {};
			\node at (-1.5,0) [infinito]  (2){$\displaystyle\infty$};
			\node at (1,0) [nodo] (3) {};
			\node at (0,1) [nodo] (4) {};
			\node at (-1.5,1) [infinito] (5) {$\displaystyle\infty$};
			\node at (2,0) [nodo] (6) {};
			\node at (3,0) [nodo] (7) {};
			\node at (2,1) [nodo] (8) {};
			\node at (3,1) [nodo] (9) {};
			\node at (4.5,0) [infinito] (10) {$\displaystyle\infty$};
			\node at (5.5,0) [infinito] (11) {$\displaystyle\infty$};
			\node at (4.5,1) [infinito] (12) {$\displaystyle\infty$};
			\draw [-] (1) -- (2) ;
			\draw [-] (1) -- (3);
			\draw [-] (1) -- (4);
			\draw [-] (3) -- (4);
			\draw [-] (5) -- (4);
			\draw [-] (3) -- (6);
			\draw [-] (6) -- (7);
			\draw [-] (6) to [out=-40,in=-140] (7);
			\draw [-] (3) to [out=10,in=-35] (1.4,0.7); 
			\draw [-] (1.4,0.7) to [out=145,in=100] (3); 
			\draw [-] (6) to [out=40,in=140] (7);
			\draw [-] (6) -- (8);
			\draw [-] (6) to [out=130,in=-130] (8);
			\draw [-] (7) -- (8);
			\draw [-] (8) -- (9);
			\draw [-] (7) -- (9);
			\draw [-] (9) -- (12);
			\draw [-] (7) -- (10); 
			\draw [-] (7) to [out=40,in=140] (11);
			\end{tikzpicture}
		\end{center}
		\caption{Example of a (noncompact) metric graph in $\mathbf{G}$ with $5$ unbounded edges and $13$ bounded edges, one of which forms a self-loop.}
		\label{graph}
	\end{figure}
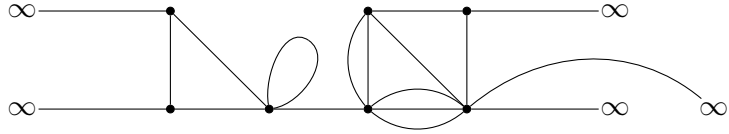

	The interest in metric graphs has been growing through the decades to become today an active research field with an inter-sectoral popularity within the scientific community. In fact, major contributions to the development of the topic stem from the fact that it attracted attention both in applied sciences and in more theoretical ones. In a wide variety of applications, indeed, metric graphs may serve as simplified models for higher dimensional branched or ramified domains, i.e. structures where the transverse dimensions are negligible with respect to the longitudinal one. At the same time, in many cases the mathematical analysis of problems on graphs proves to be of interest {\em per se}, exhibiting elements of novelty compared to standard Euclidean settings. 
	
	In the context of differential models on metric graphs, both linear and nonlinear models have been largely investigated (see e.g. \cite{BKKM,FMN,HKMP,KKLM} and references therein for some recent results in the linear setting, and the reviews \cite{ABR,KNP} for comprehensive overviews on nonlinear problems). Within the nonlinear theory, a significant attention has been focused on nonlinear dispersive equations, such as Schr\"odinger and Dirac equations, also in view of their possible application in the innovative high-tech research field of atomtronics (see \cite{atom} for a wide introduction to the subject). In particular, many efforts have been devoted to the search of bound states, i.e. solutions of the associated stationary equation (e.g. \eqref{NLS} for the nonlinear Schr\"odinger equation) coupled with various vertex conditions, unraveling a deep dependence of the problem on both topological and metric properties of graphs (see for instance \cite{ABD,acfn_aihp,ADST,ASTbound,AST1,AST,ACT,BMP,BC,BDL20,BDL21,BD1,BD2,BCJS,CJS,DDGS,DDGST,DGMP,DST20,DT,KMPX,NP,pankov,PS20,PSV} for Schr\"odinger equations and \cite{BCT1,BCT2} for Dirac equations). 
	
	The present paper fits in the investigation of bound states for nonlinear Schr\"odinger equations on metric graphs focusing specifically on the problem
	\begin{equation}
	\label{nlse}
	\begin{cases}
	- u''+\lambda u=u^{2\mu+1} & \text{ on every edge of } \G \\
	\sum_{e\succ \vv} u_e'(\vv)=0 & \text{ for every } \vv\in \mathbb V\,,
	\end{cases}
	\end{equation}
	that is \eqref{NLS} together with homogeneous Kirchhoff conditions at the vertices. Here, $u_e'(\vv)$ denotes the outgoing derivative of $u$ at $\vv$ along the edge $e$ and  $e\succ \vv$ means that the sum is extended to all edges incident at $\vv$. 
	
	The aim of this work is to prove existence and multiplicity of  $H^1$ positive bound states of \eqref{nlse} concentrating at given vertices of $\G$ as the parameter $\lambda$ goes to $+\infty$. 
	
	Before stating our main results, we need to recall what is known about positive solutions of \eqref{nlse} on  star graphs. Given any integer $N\geq1$, the $N$-star graph $\starN$ is the graph made up of a single vertex, identified with 0, and $N$ half-lines attached to it. Clearly, $\mathcal S_1=\R^+$ and $\mathcal S_2=\R$. A function $\Psi \in H^1(\starN)$ can be seen as an $N$--tuple  $(\psi_1, \dots,\psi_N)$, where $\psi_i\in H^1(\R^+)$ for every $i$ and $\psi_i(0) =\psi_j(0)$ for every $i,j$. Note that, by dilation invariance, for every $\mu>0$ and $\lambda>0$ any solution of \eqref{nlse} on $\starN$ is given by $\lambda^{\frac1{2\mu}}\Psi\big(\sqrt{\lambda}\,x\big)$, where $\Psi=(\psi_1, \dots,\psi_N)$ solves
	\begin{equation}
	\label{main}
	\begin{cases}
	-\psi_i''+\psi_i=\left|\psi_i\right|^{2 \mu} \psi_i & \text{on }\R^+,\quad\forall i=1\dots,N \\
	\sum_{i=1}^N \psi_i'(0)=0\,, & 
	\end{cases}
	\end{equation}
	namely \eqref{nlse} with $\lambda=1$.

	On the real line (i.e. $N=2$), it is well-known that the set of $L^2$ positive solutions of \eqref{main} is the family of solitons
	\begin{equation}
	\label{phi}
	\phi_a(x):=\phi(x-a),\qquad\text{with}\qquad\phi(x):=(\mu+1)^{\frac{1}{2 \mu}} \operatorname{sech}^{\frac{1}{\mu}}(\mu x)\,.
	\end{equation}
	Similarly, on the half-line $\R^+$ (that is $N=1$), problem \eqref{main} has a unique $L^2$ positive solution given by the restriction of $\phi$ to $\R^+$, named the half-soliton. For $N\geq3$, the $H^1$ solutions to problem \eqref{main} have been completely classified e.g. in \cite{acfn_pl}. Precisely, if $N$ is odd, \eqref{main} has a unique positive solution $\Psi_N = (\psi_1,\dots,\psi_N)$ given by
	\begin{equation}\label{fi}
	\psi_i (x)=\phi(x)\qquad\forall  i=1,\dots,N,
	\end{equation}
	where $\phi\in H^1(\R^+)$ is the half-soliton introduced in \eqref{phi}.
	When $N$ is even, on the contrary, \eqref{main} admits infinitely many positive solutions $\Psi_N^a = (\psi^a_1,\dots,\psi^a_N)$ (with $a\in\mathbb R$), described (after a possible permutation of indices) through $\phi_a$ in \eqref{phi} as
	\begin{equation}\label{psia}
	\psi^a_i(x) = \begin{cases} \phi_a(x) & \text{ if } i=1,\dots,N/2 \\
	\phi_{-a}(x) & \text{ if }  i=  N/2+1,\dots, N\,.
	\end{cases}
	\end{equation}
	Let us note incidentally here that, for general metric graphs, a complete description of the set of solutions of \eqref{nlse} is usually out of reach. To the best of our knowledge, besides star graphs the unique case for which this is available is the $\mathcal T$-graph (two half--lines and a bounded edge glued together at the same vertex), that has been discussed in the recent paper \cite{ACT}, whereas partial results in this direction have been given e.g. for the tadpole graph (a circle attached to an half-line) and the double-bridge graph (a circle attached to two half-lines at different points)  in \cite{NPS,NRS}.  
	
	We are now in position to state the main results of our paper, that are given in the  next two theorems.
	
	\begin{theorem}
		\label{main1} 
		Let $\G\in\mathbf{G}$ and  $\mu\ge\frac12$. Let $\overline\vv\in \V$ be a vertex of $\G$ such that $N:=\dg(\overline\vv)\geq3$ is odd. Then there exists $\lambda_0:=\lambda_0(\overline\vv)>0$ such that, for every $\lambda\geq \lambda_0$, there exists a positive $H^1$ solution $u_\lambda$  of \eqref{nlse} with a single peak at $\overline\vv$. More precisely, identifying $\overline\vv$ with $0$ along each edge $e\succ\overline\vv$ and setting $\ell_{\overline\vv}:=\min_{e\succ\overline \vv}|e|/2$, as $\lambda \to +\infty$ there holds
		\begin{equation}
		\label{u1picco}
		u_\lambda(x)= \lambda^{\frac{1}{2\mu}}\chi(x) \Psi_N\left(\sqrt{\lambda} x\right)+ \Phi(x),
		\end{equation}
		where $\chi:\G\to[0,1]$ is a smooth function satisfying $\chi\equiv 1$ on the neighbourhood of radius $\ell_{\overline\vv}$ of $\overline\vv$ in $\G$ and $\chi\equiv0$ outside the neighbourhood of radius $2\ell_{\overline\vv}$ of $\overline\vv$ in $\G$, $\Psi_N$ is given in \eqref{fi} and $\Phi$ satisfies
		\[
		\| \Phi\|_\lambda :=\sqrt{\|\Phi'\|_{L^2(\G)}^2 + \lambda  \|\Phi\|_{L^2(\G)}^2} = o\left(\lambda^{\frac14+\frac{1}{2\mu}}\right).
		\]
	\end{theorem}
	
	\begin{theorem}
		\label{main2}
		Let $\G\in\mathbf{G}$ and $\mu\geq\frac12$.  Assume that $\G$ has $M\geq2$ vertices $\overline{\vv}_1,\dots,\overline{\vv}_M\in\V$ such that $N_i:=\dg(\overline{\vv}_i)\geq3$ is odd for every $i=1,\dots,M$. Then there exists $\lambda_0:=\lambda_0(\overline{\vv}_1,\dots,\overline{\vv}_M)$ such that, for every $\lambda\geq\lambda_0$, there exists a positive $H^1$ solution $u_\lambda$ of \eqref{nlse} with a single peak at every vertex $\overline{\vv}_1,\dots,\overline{\vv}_{M}$. More precisely, identifying $\overline{\vv}_i$ with $0$ along each edge $e\succ\overline{\vv}_i$  and setting $\ell_{\overline{\vv}_i}:=\min_{e\succ\overline{\vv}_i}|e|/4$ for every $i=1,\dots,M$, as $\lambda\to+\infty$ there holds
		\begin{equation}
		\label{umulti}
		u_\lambda(x)=\lambda^{\frac1{2\mu}}\sum_{i=1}^M \chi_i\Psi_{N_i}\big(\sqrt{\lambda}\,x\big)+\Phi(x),
		\end{equation}
		where $\chi_i:\G\to[0,1]$ is a smooth function satisfying $\chi_i\equiv1$ on the neighbourhood of radius $\ell_{\overline{\vv}_i}$ of $\overline{\vv}_i$ in $\G$ and $\chi_i\equiv 0$ outside the neighbourhood of radius $2\ell_{\overline{\vv}_i}$ of $\overline{\vv}_i$ in $\G$, $\Psi_{N_i}$ is given in \eqref{fi} and $\Phi$ satisfies 
		\[
		\|\Phi\|_\lambda=o\left(\lambda^{\frac14+\frac1{2\mu}}\right)\,.
		\] 
	\end{theorem}
	The above results prove, for sufficiently large $\lambda$, existence of one-peaked (Theorem \ref{main1}) and multi-peaked (Theorem \ref{main2}) positive bound states concentrating at vertices with odd degree and being negligible (as $\lambda \to +\infty$) on the rest of the graph. The proof of both theorems is based on a Ljapunov-Schmidt procedure using as model function the solution $\Psi_N$ of \eqref{nlse} on the star graph $\starN$. Note that the notation in Theorem \ref{main2} is consistent with the fact that one can simultaneously identify all the $\overline{\vv}_i$'s with the origin along all edges incident at each of them.  This is obvious if such vertices have no common edge, whereas if two of them share one edge this can still be done by considering an additional vertex of degree $2$ at the middle point of the shared edge. Also the slightly different definition of $\ell_{\overline{\vv}_i}$ with respect to Theorem \ref{main1} is meant to allow the presence of shared edges between the $\overline{\vv}_i$'s.
	
	Note that Theorems \ref{main1}--\ref{main2} do not apply when $ \mu<\frac12$ or the vertices have even degree. The limitation on $\mu$ is technical and unavoidable with our argument, since we need to compute a second order expansion of the function $f(s)=s^{2\mu+1}$ at $s=\Psi_N$. However, our theorems cover the cubic case $\mu=1$, which is usually considered the most relevant one in many physical applications. Conversely, it is not clear to us whether one can recover the results of Theorems \ref{main1}--\ref{main2} concentrating at vertices with even degree. Heuristically, one may expect that this should be true. 
	
	Actually, the approach could be completely different. Let us focus on the simplest case of a graph with one edge and two vertices (i.e. an interval). It is clear that   there exists a solution to  \eqref{nlse} which concentrates at 
	each vertex (i.e.  boundary point) whose main profile is the function $\Psi_1$.   On the other hand it is also well known that there exists a solution  to  \eqref{nlse} which concentrates at a {\em special} point inside the edge (i.e.  the middle point of the interval) whose main profile is  the function $\Psi_1^a$ \eqref{psia} for a {\em special} value of the parameter $a$ (see for example \cite[Section 2.2]{PPVV} and the references therein). Now, we observe that each point inside the edge can be seen as a vertex of degree 2 and the previous result shows that only one of them is a peak of a concentrating solution.  It would be extremely interesting to prove that a similar result holds true for a more general graph with a  vertex of even degree. We believe that  the difference between the odd and the even case   arises in the choice of the ansatz: this is merely a cut-off of
	$\Psi_N$ around the vertex  itself in the odd case, while  in the  even case it should possibly be a global refinement of  $\Psi_N^a$ in order to fit the Kirchhoff conditions  on the whole graph (see \cite[Remark 2.16]{PPVV}). However,  at present the case of even degree is completely open.

	The interest in peaked solutions of \eqref{nlse} on metric graphs is not new. In \cite{ASTbound}, positive bound states with a maximum in the interior of a given edge are identified as solutions of a doubly constrained minimization problem, for every $\mu\in(0,2)$ and for sufficiently large masses (i.e. the $L^2$ norm of the function). In \cite{BMP,KP} a Dirichlet-to-Neumann map argument is developed in the cubic case $\mu=1$ to construct, again for large masses, solutions with maximum points either at vertices of degree 1 or in the interior of any edge. The results of these three papers apply both to compact graphs and to noncompact graphs with finitely many edges. Remarkably, both approaches are able to handle the mass constrained setting. Conversely, in \cite{DGMP} a Ljapunov-Schmidt procedure similar to the one discussed here is used to find solutions concentrating at vertices of degree 1 on compact graphs, for every $\mu>0$ and $\lambda\to+\infty$. Actually, Theorems 1.2--1.3 of \cite{DGMP} are the analogues of Theorems \ref{main1}--\ref{main2} here in the case of vertices of degree 1 and the strategy of the proof is the same. However, we stress that our results here apply to any graph in $\mathbf{G}$ and not only to compact ones. Furthermore, from the technical point of view, the Ljapunov-Schmidt argument for vertices with degree greater than 1 is rather different and technically demanding. This is readily seen observing that linearizing \eqref{nlse} around $\Psi_N$ gives a linear problem that has only the trivial solution when $N=1$, whereas it has nontrivial solutions as soon as $N\geq2$ (see Lemma \ref{non} below).
	
	With respect to the available literature on bound states of \eqref{nlse}, the main novelty of Theorems \ref{main1}--\ref{main2} is that these are the first results exhibiting solutions with maximum points at vertices of degree greater than or equal to 3. To better understand why this is worth noting, we recall that almost all the existence results on general metric graphs derived so far are based on minimization arguments. In particular, major attention has been devoted to minimum problems both for the action functional $J_\lambda:H^1(\G)\to\R$
	\[
	J_\lambda(u):=\frac12\|u'\|_{L^2(\G)}^2+\frac\lambda2\|u\|_{L^2(\G)}^2-\frac1{2\mu+2}\|u\|_{L^{2\mu+2}(\G)}^{2\mu+2}
	\] 
	constrained to the associated Nehari manifold
	\[
	\mathcal{N}_\lambda(\G):=\left\{u\in H^1(\G)\,:\,J_\lambda'(u)u=0\right\},
	\]
	and for the energy functional $E:H^1(\G)\to\R$
	\[
	E(u):=\frac12\|u'\|_{L^2(\G)}^2-\frac1{2\mu+2}\|u\|_{L^{2\mu+2}(\G)}^{2\mu+2}
	\]
	constrained to the space of functions with prescribed mass
	\[
	H_\nu^1(\G):=\left\{u\in H^1(\G)\,:\,\|u\|_{L^2(\G)}^2=\nu\right\}.
	\]
	Thorough investigations have been developed for ground states, i.e. global minimizers of both problems (see e.g. \cite{DDGS,pankov} for the action and \cite{ADST,AST,AST1,ASTcmp,DST20,DT} and references therein for the energy), but local minimizers have been investigated too (see e.g. \cite{ASTbound,PSV}). However, none of these solutions coincides with those in Theorems \ref{main1}--\ref{main2}. 
	
	\begin{corollary}
		\label{cor1}
		Let $u_\lambda$ be as in Theorems \ref{main1}--\ref{main2}. Then $u_\lambda$ is neither a ground state of $J_\lambda$ in $\mathcal{N}_\lambda(\G)$ nor a ground state of $E$ in $H_\nu^1(\G)$ (where $\nu = \|u_\lambda\|_{L^2(\G)}^2$).
	\end{corollary}
	The proof of Corollary \ref{cor1}, which is given for the sake of completeness in Section \ref{sec:cor}, is actually a straightforward consequence of the fact that, by \eqref{u1picco}, \eqref{umulti}, one can compute explicitly the action, the energy and the mass of $u_\lambda$, and a direct comparison with the asymptotic behaviour of the ground state levels shows that $u_\lambda$ is always a solution with action/energy strictly larger than that of ground states. This is no surprise, as ground states are known to concentrate at vertices of degree 1 (when available) or in the interior of the edges (see e.g. \cite{BMP,DGMP}).
	
	Hence, Theorems \ref{main1}--\ref{main2} provide genuinely new existence and multiplicity results. As for the mass of these solutions, observe that, if $u_\lambda$ is as in \eqref{u1picco}, then
	\begin{equation}
	\label{massa1}
	\|u_\lambda\|_{L^2(\G)}^2=\lambda^{\frac1\mu-\frac12}\left(\frac N2\|\phi\|_{L^2(\R)}^2+o(1)\right)\,,
	\end{equation}
	whereas if $u_\lambda$ is as \eqref{umulti}
	\begin{equation}
	\label{massamulti}
	\|u_\lambda\|_{L^2(\G)}^2=\lambda^{\frac1\mu-\frac12}\left(\sum_{i=1}^M\frac {N_i}2\|\phi\|_{L^2(\R)}^2+o(1)\right).
	\end{equation}
	In particular, these are bound states with diverging masses in the $L^2$-subcritical regime $\mu<2$, with masses strictly greater than $\|\phi\|_{L^2(\R)}^2$ at the $L^2$-critical power $\mu=2$, and with vanishing masses in the $L^2$-supercritical regime $\mu>2$. This is of particular interest both for $\mu=2$, since in the critical regime it is usually difficult to find solutions with masses larger than $\|\phi\|_{L^2(\R)}^2$ (see e.g. \cite{ASTcmp,PSV}), and for $\mu>2$, as very few results are available at present in the $L^2$-supercritical case.
	
	\begin{figure}[t]
		\centering
		\subfloat[][ ]{\includegraphics[width=0.35\columnwidth]{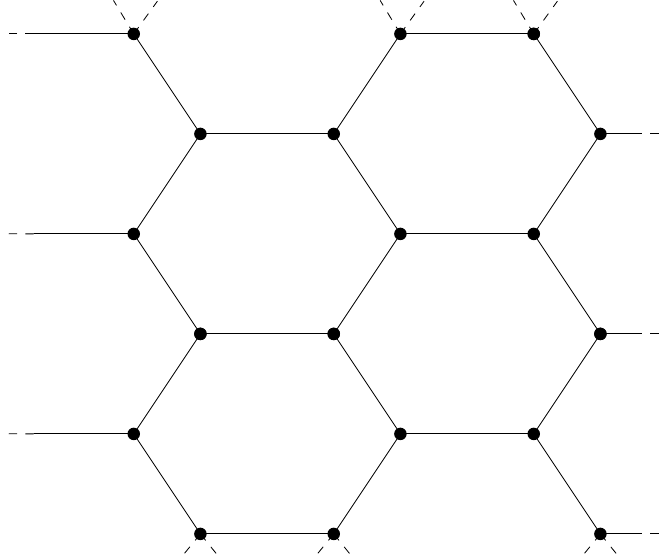}}
		\qquad\qquad
		\subfloat[][ ]{\begin{tikzpicture}[xscale= .8,yscale=.8]
			\begin{scope}[xshift=17em,grow cyclic,shape=circle, thick,
			level distance=2.7em,
			cap=round]
			\tikzstyle{level 1}=[rotate=-90,sibling angle=120]
			\tikzstyle{level 2}=[sibling angle=85]
			\tikzstyle{level 3}=[sibling angle=57]
			\tikzstyle{level 4}=[sibling angle=60]
			\tikzstyle{edge from parent}=[draw]
			\node at (0,0) [nodo] {} child  foreach \A in {1,2,3}
			{ node [nodo]
				{} child  foreach \B in {1,2}
				{ node [nodo] {} child   foreach \C in {1,2}
					{ node [nodo] {} child [dashed,level distance=1.5em]  foreach \C in {1,2} }
				}
			};
			\end{scope}
			\end{tikzpicture}}
		\caption{Examples of infinite periodic graphs (A) and infinite trees (B) with vertices of odd degree greater than 1.}
		\label{fig:inf}
	\end{figure}
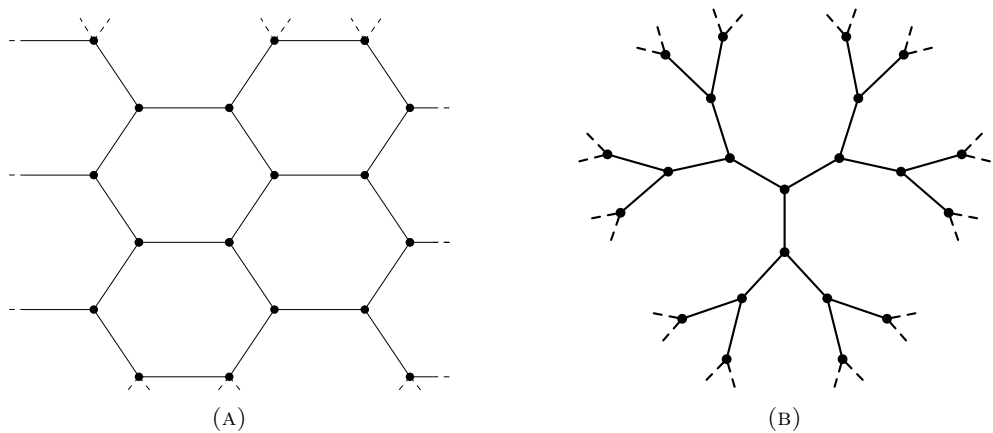

	To conclude, we further observe that the results of this paper apply to graphs with countably many edges. As so, Theorem \ref{main2} has evident consequences on the set of bound states of \eqref{nlse} on graphs of this type, such as infinite periodic graphs (Figure \ref{fig:inf}(A)) and infinite trees (Figure \ref{fig:inf}(B)). In particular, graphs like these admit at least countably many $H^1$ positive solutions of \eqref{nlse}, each one with an arbitrary number of peaks located at any given subset of vertices with odd degree. Furthermore, in view of \eqref{massamulti}, considering a sufficiently large number of vertices one obtains bound states with arbitrarily large mass, and this is remarkably true independently of $\mu\geq\frac12$.
	
	\smallskip
	The remainder of the paper is organized as follows. Section \ref{sec:prel} collects some preliminaries. Sections \ref{sec:phi}--\ref{sec:bi}--\ref{sec:G} provide the proof of Theorem \ref{main1}, with the reduction to a finite dimensional problem (Section \ref{sec:phi}), its formulation in term of a reduced energy (Section \ref{sec:bi}) and the analysis of the critical points of such energy (Section \ref{sec:G}). Finally, Section \ref{sec:multi} discusses the proof of Theorem \ref{main2} and Section \ref{sec:cor} that of Corollary \ref{cor1}.
	
	\medskip
	{\em Notation.} In the following we will write $f\lesssim g$ or $f=\mathcal O(g)$ in place of $|f|\le C |g| $  for some positive constant $C$ independent of $\lambda$ and $f\sim g$  in place of $f= g+o(g)$, whenever possible. Furthermore, except for Section \ref{sec:multi}, we will always write $\Psi$ in place of $\Psi_N$.

	\section{Preliminaries}
	\label{sec:prel}
	
	We begin  by introducing the basic idea to construct the solutions to \eqref{nlse} we are looking for and by collecting some related preliminary results.
	
	Given a vertex $\overline\vv$ of $\G$ with odd degree greater than or equal to 3, our aim is to find, for suitable values of the parameter $\lambda$, solutions $u_\lambda$ to \eqref{nlse} in the form
	\begin{equation}
	\label{eq:ansatz}
	u_\lambda:=W_\lambda+\Phi\,,
	\end{equation}
	where $W_\lambda$ concentrates at $\overline\vv$ and $\Phi$ is a smaller order term as $\lambda\to+\infty$. To do this, we first identify a good candidate for the principal part $W_\lambda$ and then derive the correction term $\Phi$ with a Ljapunov-Schmidt procedure.
	
	To define $W_\lambda$, we start with the unique symmetric solution $\Psi\in H^1(\starN)$ to \eqref{main} on the $N$--star graph $\starN$, i.e. $\Psi = (\phi,\dots,\phi)$, where $\phi\in H^1(\R^+)$ is the half--soliton on $\R^+$, and we characterize the set of solutions of the linearization of \eqref{main} at $\Psi$.
	
	\begin{lemma}
		\label{non}
		For every $N\geq2$, the set of solutions to 
		\begin{equation}
		\label{linear}
		\begin{cases} -Z''+  Z = (2\mu+1) \Psi ^{2 \mu}   Z, & Z = (z_1,\dots,z_N) \in H^1(\starN) \\
		\sum_{i = 1}^N z_i'(0)=0 &
		\end{cases}
		\end{equation}
		is given by the $(N-1)$--dimensional space
		\[
		K:=\text{\normalfont span}\left\{Z^{(j)}=\left(Z_1^{(j)},\dots, Z_N^{(j)}\right)\,:\,j=1,\dots,N-1\right\}\,,
		\]
		where, for every $j=1,...,N-1$,
		\begin{equation}
		\label{zj}
		Z_i^{(j)}(x)=\phi'(x)e^j\qquad \forall x\in\R^+,\,i=1,\dots,N,
		\end{equation}
		and
		\begin{equation}
		\label{ej} 
		e^j=\left(\begin{matrix}e^j_1\\ e^j_2 \\ \vdots \\ e^j_N\end{matrix}\right),\quad\hbox{ with }  e^j_1+...+e^j_N=0\quad\forall j \quad \hbox{and}\quad \ e^j\cdot e^k =0\text{ if } j\ne k.
		\end{equation}
	\end{lemma}
	
	\begin{proof}
		Since $\Psi=(\phi, \dots,\phi)$, we plainly see that on each half-line of $\starN$, the function $\phi'$  solves the differential equation in \eqref{linear}. Hence, by the standard theory of linear ordinary differential equations, the general solution of the first line of \eqref{linear} on the $i$--th half-line is of the form
		\[
		z_i(x)=\phi'(x)\(c_i+d_i\int_{x_0}^x \frac{1}{\phi'(s)^2}ds\)
		\]
		with $c_i, d_i \in \R$.
		Since $\phi'$  decays exponentially as $x\to +\infty$, it follows that $z_i\in L^2(\mathbb R^+)$ if and only if $d_i=0$, in which case
		\[
		z_i (x)= c_i\phi' (x).
		\]
		As a consequence, the function $Z= ( c_1\phi', \dots,  c_N\phi')$ belongs to $H^1(\starN)$, as every component is in $H^1(\R^+)$ and continuity at the vertex is guaranteed for every values of the $c_i$'s because $\phi'(0) = 0$. Therefore $Z$ solves problem \eqref{linear} if and only if the homogeneous Kirchhoff condition is satisfied, namely if and only if 
		\[
		0= \sum_{i=1}^N z_i'(0)=  \sum_{i=1}^N c_i\phi''(0).
		\]
		As $\phi''(0)\not=0$, this yields 
		$c_1+...+c_N=0$. Since the equation $c_1+ \dots + c_N =0$ identifies an $(N-1)$--dimensional subspace of $\R^N$ a basis of which is given e.g. by vectors $e^j$ as in \eqref{ej}, this concludes the proof.
	\end{proof}
	For the sake of convenience, in the following we will choose the vectors $e^j, j=1,...,N-1$, satisfying \eqref{ej} to be
	\begin{equation}
	\label{uj} 
	e^1=\left(\begin{matrix} 1\\ -1\\ 0\\ \vdots\\ 0 \end{matrix} \right) ,\ 
	e^2=\left(\begin{matrix} 1\\ 1\\ -2\\ \vdots\\ 0 \end{matrix} \right) ,\ \dots, \ 
	e^{N-1}=\left(\begin{matrix} 1\\  1\\ 1\\ \vdots\\ -(N-1) \end{matrix} \right) .
	\end{equation}
	
	We are now in position to define $W_\lambda$. Here we denote by $B(a,r)$ the ball of radius $r$ in $\G$ centered at the point $a$. Set $N:=\dg(\overline{\vv})$, $\ell:=\min_{e\succ \overline\vv}|e|/2$ and let $\chi:\G\to[0,1]$ be a smooth cut-off function such that $\chi\equiv1$ on $B(\overline\vv,\ell)$ and $\chi\equiv0$ on $\G\setminus B(\overline\vv, 2\ell)$. We then define $W_\lambda:\G\to\R$ as 
	\[
	W_\lambda(x):=\chi(x)\left(\Psi_\lambda(x)+b_{1,\lambda}Z_\lambda^{(1)}(x)+...+b_{N-1,\lambda}Z_\lambda^{(N-1)}(x)\right)
	\]
	where
	\[
	\Psi_\lambda(x):= \lambda^{\frac{1}{2\mu}}\Psi\big(\sqrt{\lambda}\,x\big)
	\]
	with $\Psi$ the unique symmetric solution to \eqref{main} as above, 
	\[
	Z^{(j)}_\lambda(x):= \lambda^{\frac{1}{2\mu}}Z^{(j)}\big(\sqrt{\lambda}\,x\big)
	\]
	with $Z^{(j)}$ the functions in Lemma \ref{non}, and the real numbers $b_{j,\lambda}$  given by
	\begin{equation}
	\label{bj}
	b_{j,\lambda}=\frac{\mathfrak b_j}{\lambda^\alpha}
	\end{equation}
	for suitably chosen $\mathfrak b_j\in\mathbb R$ and $\alpha>0$. Here, even though in principle $\mathfrak{b}_j$ may still depend on $\lambda$, we decide not to denote explicitly such  dependence since, as will be clear in the next sections, the actual value of the $\mathfrak{b}_j$'s we will consider remains always uniformly bounded  in $\lambda$.
	
	With this definition, to prove Theorem \ref{main1} we need to find $\Phi$ and $\mathfrak{b}_j$, $j=1,\dots,N-1$, so that $u_\lambda$ as in \eqref{eq:ansatz} solves \eqref{nlse}.  Note that, setting $f(u):=(u^+)^{2\mu+1}$ for every $u\in H^1(\G)$,  it is clear that positive solutions of \eqref{nlse} coincide with solutions of 
	\begin{equation}\label{pro}
	\begin{cases}
	-u''+\lambda u=f(u) & \text{on }\G\\
	\sum_{e\succ \vv}u_e'(\vv)=0 &  \text{for every }\vv\in \V.
	\end{cases} 
	\end{equation}
	Since we will consider the limit $\lambda\to+\infty$, with no loss of generality we can assume from the beginning $\lambda>-\lambda_\G$, where 
	\[
	\lambda_\G:=\inf_{v\in H^1(\G)}\frac{\|v'\|_{L^2(\G)}^2}{\|v\|_{L^2(\G)}^2}
	\]
	denotes the bottom of the spectrum of $-d^2/dx^2$ (coupled with homogeneous Kirchhoff conditions) on $\G$. Hence, for every such $\lambda$ we equip the space $H^{1}(\G)$ with
	the following equivalent scalar product
	\[
	\left\langle u,v\right\rangle _{\lambda}=\int_{\G}u'(x)v'(x)dx+\lambda\int_{\G}u(x)v(x)dx
	\]
	and denote by $\|\cdot\|_{\lambda}$ the corresponding norm. Note that, considering the immersion
	
	\[
	i_{\lambda}:\left(H^{1}(\G),\left\langle \;,\;\right\rangle_{\lambda}\right)\rightarrow\left(L^{2}(\G),\left\langle\; ,\;\right\rangle _{L^{2}}\right)
	\]
	and defining as usual its adjoint map
	\[
	i_{\lambda}^{*}:\left(L^{2}(\G),\left\langle\; ,\;\right\rangle _{L^{2}}\right)\rightarrow\left(H^{1}(\G),\left\langle \;,\;\right\rangle _{\lambda}\right)
	\]
	so that 
	\[
	\left\langle i_{\lambda}^{*}(g),v\right\rangle _{\lambda}=\left\langle g,v\right\rangle _{L^{2}}\qquad\forall v\in H^{1}(\G),\; g\in L^2(\G),
	\]
	we  obtain that
	\[
	v=i_{\lambda}^{*}(g)\quad\Longleftrightarrow\quad v\text{ solves }\left\{ \begin{array}{ll}
	-v''+\lambda v=g & \text{on }\G\\
	\sum_{e\succ \vv}v_e'(\vv)=0 & \text{for every }\vv\in V\,.
	\end{array}\right.
	\]
	Therefore, problem \eqref{pro} can be rewritten as
	\begin{equation}\label{pro2}
	u=i_{\lambda}^{*}\(f(u)\),\ u\in H^1(\G)
	\end{equation}
	and to find $u_\lambda$ as in \eqref{eq:ansatz} solving \eqref{nlse} amounts to find $\Phi$, $\mathfrak{b}_j$, $j=1,\dots, N-1$, such that $u_\lambda$ satisfies \eqref{pro2}. Actually, we will further rewrite \eqref{pro2} as follows. For every $\lambda$, we introduce the linear operator $\LL:H^1(\G)\to H^1(\G)$
	\begin{equation}
	\label{lp}
	\LL(v):=v-i^*_\lambda\(f'(W_\lambda)v\),
	\end{equation}
	the nonlinear operator $\NN: H^1(\G)\to H^1(\G)$
	\begin{equation}
	\label{np}
	\NN(v):=i^*_\lambda\(f(W_\lambda+v) -f(W_\lambda)-f'(W_\lambda)v\),
	\end{equation}
	and the error term
	\begin{equation}
	\label{e}
	\EE:= i^*_\lambda\(f(W_\lambda)\)-W_\lambda\,.
	\end{equation}
	Accordingly, \eqref{pro2} is equivalent to 
	\begin{equation}
	\label{pro3}
	\LL(\Phi)=\EE+\NN(\Phi), 
	\end{equation}
	where of course one still needs to identify both $\Phi$ and the coefficients $\mathfrak{b}_j$'s.

	\begin{remark}
		\label{rem:i*}
		Note that, if $\G$ is a noncompact graph, $i_\lambda^*$ is not compact. However, since by definition $W_\lambda$ is compactly supported in a fixed ball centered at $\overline{\vv}$, the operator from $\left(L^2(\G), \langle\,,\,\rangle_{L^2}\right)$ to $\left(H^1(\G), \langle\,,\,\rangle_\lambda\right)$ given by $v\mapsto i_\lambda^*(f'(W_\lambda)v)$  is compact for every given $\lambda$. As a consequence, the operator $\LL$ in \eqref{lp} is a compact perturbation of the identity for every $\lambda$.
	\end{remark}
	
	\begin{remark}
		Since we will need them in the following, we conclude this section collecting here the following elementary inequalities
		\begin{equation}\label{in1}
		\left||a+b|^q-|a|^q\right|\lesssim  a^{q-1}|b|+|b|^q \qquad\forall q>1\,,
		\end{equation}
		\begin{equation}\label{in2}
		\left||a+b|^{p}-a^{p}-pa^{p-1}|b|\right|\lesssim a^{p-2}|b|^2+|b|^{p},\qquad\forall p>2
		\end{equation}
		and
		\begin{equation}\label{in3}
		\left||a+b|^{r}-a^{r}-ra^{r-1}|b|-\frac12r(r-1)a^{r-2}b^2\right|\lesssim a^{r-3}|b|^3+|b|^{r}\qquad\forall r>3\,.
		\end{equation}
	\end{remark}
	
	\section{The finite--dimensional reduction}
	\label{sec:phi}
	In this and in the next two sections we solve problem \eqref{pro3}. Here we start by showing that, once the value of the $\mathfrak{b}_i$'s is fixed, it is possible to find a unique $\Phi$ fulfilling a slightly modified version of \eqref{pro3}. This will reduce the problem to identify the $\mathfrak{b}_j$'s  to solve \eqref{pro3}, a task that will be accomplished in the following two sections.
	
	To deal with the problem for $\Phi$, we introduce for every $\lambda$ the following spaces
	\[
	\begin{split}
	&K _{\lambda}:=\left\{v\in H^1(\G)\,:\,v(x)=\chi(x)\sum\limits_{j=1}^{N-1}c_j Z_\lambda^{(j)}(x)\quad\text{for some }\ c_j\in \mathbb R, \, j=1,\dots,N-1\right\}\\
	&K^\perp_{\lambda}:=\left\{v\in H^1(\G)\,:\, \langle v, \chi Z^{(j)}_\lambda\rangle_\lambda=0,\quad\forall\ j=1,\dots,{N-1}\right\}
	\end{split}
	\]
	and the corresponding projections $\Pi _\lambda:H^1(\G)\to  K _{\lambda}$, $\Pi^\perp_\lambda:H^1(\G)\to  K^\perp_{\lambda}$.  
	
	We can then state the main result of this section.
	
	\begin{proposition}
		\label{resto}
		For every compact subset $C$ in $\mathbb R^{N-1}$, there exists $\lambda_0>0$ (depending on $C$) such that, for every  $(\mathfrak b_1,\dots,\mathfrak b_{N-1})\in C$ and  every  $\lambda>\lambda_0$,
		there exist unique $\Phi\in K^\perp_\lambda$ and coefficients $c_1,\dots, c_{N-1}\in\R$ (depending on $\Phi$) for which 
		\begin{equation}
		\label{lhe}
		\LL(\Phi)=\EE+\NN(\Phi)+\chi\sum_{i=1}^{N-1}c_i  Z^{(i)}_\lambda\,.
		\end{equation}
		Moreover, 
		\begin{equation}
		\label{rate}
		\|\Phi\|_\lambda\lesssim  \lambda^{\frac14+\frac1{2\mu}-2\alpha}\,,
		\end{equation}
		where $\alpha$ is the exponent introduced in \eqref{bj}.
	\end{proposition}
	
	To prove Proposition \ref{resto}, we need the following preliminary lemma.
	
	\begin{lemma}
		\label{lem:invL}
		For every compact subset $C$ of $\mathbb R^{N-1}$, there exists $\lambda_0>0$ (depending on $C$) such that, for every  $(\mathfrak b_1,\dots,\mathfrak b_{N-1})\in C$ and  every  $\lambda>\lambda_0$, the linear operator $\Pi_\lambda^\perp\LL: K_\lambda^\perp\to K_\lambda^\perp$ is invertible with continuous inverse.
	\end{lemma}
	
	\begin{proof}  
		Note first that to prove the claim it is enough to show that there exist $c>0$ and $\lambda_0>0$ such that, for every $\lambda>\lambda_0$ and every $v\in K^\perp_\lambda$, there holds
		\begin{equation}
		\label{inj}
		\|\Pi^\perp_\lambda \LL(v)\|_\lambda\ge c\|v\|_\lambda.
		\end{equation}
		Indeed, as $v \in K^\perp_\lambda$,
		\[
		\Pi^\perp_\lambda \LL(v) = v - \Pi^\perp_\lambda i_{\lambda}^*(f'(W_\lambda)v)
		\]
		which is a compact perturbation of the identity as pointed out in Remark \ref{rem:i*}.
		
		To prove \eqref{inj}, we argue similarly to \cite[Lemma 3.1]{DGMP}. Suppose for contradiction that \eqref{inj} is false, namely that there exist sequences  $\lambda_n\to+\infty$ and $v_n\in K^\perp_{\lambda_n}$ such that, as $n \to \infty$,
		\begin{equation}
		\label{assurdo}
		\|v_n\|_{\lambda_n}=1, \qquad \|\Pi^\perp_{\lambda_n} \LL(v_n)\|_{\lambda_n} \to 0.
		\end{equation}
		We then write
		\[
		i_{\lambda_n}^*(f'(W_{\lambda_n})v_n) = v_n  -  \LL(v_n) = v_n - \Pi_{\lambda_n} \LL(v_n) -\Pi^\perp_{\lambda_n} \LL(v_n)=: v_n +k_n - h_n,
		\]
		with $k_n \in K_{\lambda_n}$ and $\|h_n\|_{\lambda_n} \to 0$ by assumption. By definition of $i_{\lambda_n}^*$, on every edge of $\G$ this reads
		\begin{equation}
		\label{eqdiff}
		-( v_n +k_n - h_n)'' + \lambda_n ( v_n + k_n - h_n)  = f'(W_{\lambda_n})v_n,
		\end{equation}
		together with homogeneous Kirchhoff conditions at every vertex of $\G$. The rest of the proof is divided in two steps.
		
		\smallskip
		{\em Step 1.} We claim that $\|k_n \|_{\lambda_n} \to 0$. Indeed, for some numbers $c_{i, n}$,
		\[
		k_n (x)=	\chi(x)\sum_{i=1}^{N-1} c_{i, n} Z_{\lambda_n}^{(i)}(x),\qquad Z_{\lambda_n}^{(i)}(x)=\lambda_n^\frac1{2\mu}Z^{(i)}\big(\sqrt{\lambda_n}\,x\big).
		\]
		For fixed $j=1,\dots, N-1$, we multiply \eqref{eqdiff} by $\chi  Z_{\lambda_n}^{(j)}$ and, recalling that $v_n- h_n\in K^\perp_{\lambda_n}$, we obtain
		\begin{equation}
		\label{AB}
		A_n:= \int_\G f'(W_{\lambda_n})v_n\chi  Z_{\lambda_n}^{(j)} \,dx = \sum_{i=1}^{N-1} c_{i,n} \int_\G\big(\chi Z_{\lambda_n}^{(i)}\big)'\big(\chi Z_{\lambda_n}^{(j)}\big)'+\lambda_n\big( \chi Z_{\lambda_n}^{(i)}\big)\big(\chi Z_{\lambda_n}^{(j)}\big)\,dx =: B_n\,.
		\end{equation}
		Now, as $ Z_{\lambda_n}^{(j)}$ and $ Z_{\lambda_n}^{(i)} $ are pointwise orthogonal by \eqref{zj}, \eqref{ej},
		\begin{equation}
		\label{Bn}
		B_n = 
		c_{j,n} \int_{\G \cap B(\overline\vv,2\ell)}\ |\chi|^2  \big|\big(Z_{\lambda_n}^{(j)}\big) ' \big|^2 + \lambda_n |\chi|^2 \big|Z_{\lambda_n}^{(j)}\big |^2+ 2\chi\chi' \big(Z_{\lambda_n}^{(j)}\big) ' Z_{\lambda_n}^{(j)} +|\chi'|^2 \big|Z_{\lambda_n}^{(j)}\big|^2\,dx.
		\end{equation}
		As $n\to \infty$, a direct computation shows that
		\begin{equation}
		\label{est1}
		\begin{split}
		\int_{\G \cap B(\overline\vv,2\ell)} & |\chi|^2  \big|\big(Z_{\lambda_n}^{(j)}\big) ' \big|^2 \dx = \lambda_n^{\frac1\mu +1}  \int_{\G \cap B(\overline{\vv},2\ell)}\ |\chi(x)|^2 \big|\big(Z^{(j)}\big)'\big(\sqrt{\lambda_n}\,x\big)\big|^2 \dx \\
		&= \lambda_n^{\frac1\mu +\frac12}  \int_{\starN \cap B(0,2\ell\sqrt{\lambda_n})} |\chi(x/\sqrt{\lambda_n}) |^2 \big|\big(Z^{(j)}\big)'(x)\big|^2 \dx
		\sim a \lambda_n^{\frac1\mu +\frac12}
		\end{split}
		\end{equation}
		for some constant $a>0$. Note that in the previous computation we tacitly interpreted the term $\chi \big(Z^{(j)}\big)'$ as defined both on a compact subset of $\G$ and on a compact subset of $\starN$. This is clearly unambiguous, since the support of the cut-off function $\chi$ is a finite symmetric star graph with $N$ edges centered at $\overline\vv$ (recall that here $N=\dg(\overline{\vv})$). The identification of such subsets of $\G$ and $\starN$ will be frequently used also in the rest of the proof. 
		
		Arguing as in \eqref{est1} one easily sees that the second term in \eqref{Bn} has the same asymptotic behavior, while the last two terms are $o\left( \lambda_n^{\frac1\mu +\frac12}\right)$. Therefore,
		\begin{equation}
		\label{B}
		B_n \sim ac_{j,n} \lambda_n^{\frac1\mu +\frac12}\qquad\text{as }n\to+\infty\,.
		\end{equation}
		Let us now focus on $A_n$. Since each $Z^{(j)}$ solves \eqref{linear}, the functions $Z_{\lambda_n}^{(j)}$ solve 
		\[
		- \big(Z_{\lambda_n}^{(j)}\big)''+\lambda_n Z_{\lambda_n}^{(j)}=f'(\Psi_{\lambda_n}) Z_{\lambda_n}^{(j)}
		\]
		on every edge of $\starN$, coupled with homogenous Kirchhoff conditions at the vertex. Multiplying by $\chi v_n$, thinking of the resulting equation as defined on $\G$ (due to the cut-off $\chi$) and integrating on $\G$ yields
		\begin{align*}
		\int_\G f'(\Psi_{\lambda_n})Z_{\lambda_n}^{(j)} \chi v_n\,dx&=\int_\G\big(Z_{\lambda_n}^{(j)}\big)'(\chi v_n)'+\lambda_n Z_{\lambda_n}^{(j)} \chi v_n\,dx\\
		&=\int_\G \big(\chi Z_{\lambda_n}^{(j)}\big)'v_n'+\lambda_n Z_{\lambda_n}^{(j)}\chi v_n\,dx+\int_\G\chi'\Big(\big(Z_{\lambda_n}^{(j)})' v_n-v_n'Z_{\lambda_n}^{(j)}\Big)\,dx.
		\end{align*}
		Notice that the first integral in the right hand side vanishes since $v_n \in K_{\lambda_n}^\perp$, while 
		\begin{align*}
		\left| \int_\G\chi'\Big(\big(Z_{\lambda_n}^{(j)})' v_n-v_n'Z_{\lambda_n}^{(j)}\Big)\,dx\right | &\le \|v_n\|_{L^2(\G)} \left(\int_\G |\chi'|^2 \big|\big(Z_{\lambda_n}^{(j)}\big)'\big|^2\,dx\right)^{\frac12} + \|v_n'\|_{L^2(\G)} \left(\int_\G |\chi'|^2 \big|Z_{\lambda_n}^{(j)}\big|^2\,dx\right)^{\frac12} \\
		& \lesssim \lambda_n^{\frac1{2\mu} +\frac14} \left(\int_{\ell\sqrt{\lambda_n}}^{2\ell\sqrt{\lambda_n}} |\phi''|^2\,dx\right)^{\frac12} +
		\lambda_n^{\frac1{2\mu} -\frac14} \left(\int_{\ell\sqrt{\lambda_n}}^{2\ell\sqrt{\lambda_n}} |\phi'|^2\,dx\right) ^{\frac12}\,,
		\end{align*}
		where $\phi$ is the soliton on $\R$ as in \eqref{phi}. Since $\phi'$ and $\phi''$ decay exponentially as $x\to+\infty$, there exists $\beta>0$ such that as $n\to+\infty$
		\[
		\left |\int_\G f'(\Psi_{\lambda_n})Z_{\lambda_n}^{(j)} \chi v_n \,dx\right | \lesssim e^{-\beta \sqrt{\lambda_n}}\,.
		\]
		Hence, recalling the definition of $A_n$ and \eqref{assurdo},
		\begin{equation} 
		\label{An}
		\begin{split}
		A_n &=\int_\G \(f'(W_{\lambda_n})-f'(\Psi_{\lambda_n})\)v_n\chi Z_{\lambda_n}^{(j)}\,dx+\int_\G f'(\Psi_{\lambda_n}) Z_{\lambda_n}^{(j)}\chi v_n\,dx\\
		& = \int_{\G\cap B(\overline\vv,2\ell)} \(f'(W_{\lambda_n})-f'(\Psi_{\lambda_n})\)v_n\chi Z_{\lambda_n}^{(j)}\,dx + O\left(e^{-\beta \sqrt{\lambda_n}}\right)\\
		&\le \lambda_n^{-1/2} \left\| \(f'(W_{\lambda_n})-f'(\Psi_{\lambda_n})\) \chi Z_{\lambda_n}^{(j)} \right\|_{L^2(\G\cap B(\overline\vv,2\ell))} + O\left(e^{-\beta \sqrt{\lambda_n}}\right)\,.
		\end{split}
		\end{equation}
		Neglecting for a while the last term and recalling that $f(s)=(s^+)^{2\mu+1}$ we see that 
		\begin{equation}
		\label{L2}
		\begin{split}
		&\int_{\G\cap B(\overline\vv,2\ell)}\left|\(f'(W_{\lambda_n})-f'(\Psi_{\lambda_n})\) \chi Z_{\lambda_n}^{(j)} \right|^2\,dx\lesssim\int_{\G\cap B(\overline\vv,2\ell)}\left|W_{\lambda_n}^{2\mu}-\Psi_{\lambda_n}^{2\mu}\right|^2\left|Z_{\lambda_n}^{(j)}\right|^2\,dx\\
		&\, =\lambda_n^{2+\frac1\mu}\int_{\starN\cap B(0,2\ell)}\left|\chi^{2\mu}(x)\left(\Psi\big(\sqrt{\lambda_n}\,x\big)+\sum_{i=1}^{N-1}b_{i,\lambda_n}Z^{(i)}\big(\sqrt{\lambda_n}\,x\big)\right)^{2\mu}-\Psi^{2\mu}\big(\sqrt{\lambda_n}\,x\big)\right|^2\\
		&\qquad\qquad\qquad\qquad\qquad\qquad\qquad\qquad\qquad\qquad\qquad\qquad\qquad\qquad\qquad\qquad\qquad\left|Z^{(j)}\big(\sqrt{\lambda_n\,x}\big)\right|^2\,dx\\
		&\lesssim \lambda_n^{\frac32+\frac1{\mu}}\int_{\starN}\left|\left(\Psi(x)+\sum_{i=1}^{N-1}b_{i,\lambda_n}Z^{(i)}(x)\right)^{2\mu}-\Psi^{2\mu}(x)\right|^2\left|Z^{(j)}(x)\right|^2\,dx\,.
		\end{split}
		\end{equation}
		Since $Z^{(i)}\in L^\infty(\starN)$ by Lemma \ref{non}, $b_{i,\lambda_n}=\lambda^{-\alpha}\mathfrak{b}_i$ by \eqref{bj}, and $(\mathfrak{b}_1,\dots,\mathfrak{b}_{N-1})\in C$ which is compact by assumption,  the mean value theorem yields the existence of $\theta_n:\starN\to[0,1]$ so that as $n\to+\infty$
		\[
		\begin{split}
		&\int_{\starN}\left|\left(\Psi+\sum_{i=1}^{N-1}b_{i,\lambda_n}Z^{(i)}\right)^{2\mu}-\Psi^{2\mu}\right|^2\left|Z^{(j)}\right|^2\,dx\\
		&\,\lesssim\lambda_n^{-2\alpha}\int_{\starN}\left|\Psi+\theta_n\sum_{i=1}^{N-1}b_{i,\lambda_n}Z^{(i)}\right|^{2(2\mu-1)}\left|\sum_{i=1}^{N-1}\mathfrak{b}_i Z^{(i)}\right|^2\left|Z^{(j)}\right|^2\,dx\\
		&\,\lesssim \lambda_n^{-2\alpha}\int_{\R^+}\left(|\phi|+|\phi'|\right)^{2(2\mu-1)}|\phi'|^4\,dx\sum_{i=1}^{N-1}\mathfrak{b}_i^2\lesssim \lambda_n^{-2\alpha}\,,
		\end{split}
		\]
		the last inequality using also $\mu\geq\frac12$. Coupling with \eqref{An} and \eqref{L2} entails
		\[
		A_n\lesssim \lambda_n^{\frac14+\frac1{2\mu}-\alpha}\quad\text{as }n\to+\infty
		\]
		and combining with \eqref{AB} and \eqref{B} we obtain
		\[
		c_{j,n}\lesssim \lambda ^{-\frac14-\frac1{2\mu}-\alpha}\quad\text{as }n\to+\infty\,.
		\]
		Finally, with the same argument used to compute $B_n$,
		\begin{equation}
		\label{knto0}
		\|k_n\|_{\lambda_n}^2=\sum_{i=1}^{N-1} c_{i,n}^2 \int_\G  \big|\big(\chi Z_{\lambda_n}^{(i)}\big)'\big|^2+\lambda_n \big|\chi Z_{\lambda_n}^{(i)}\big|^2\,dx \sim
		\lambda_n^{\frac1\mu +\frac12}\sum_{i=1}^{N-1} c_{i,n}^2\lesssim \lambda_n^{-2\alpha}\to0
		\end{equation}
		as $n\to+\infty$ since $\alpha > 0$.
		
		\smallskip
		{\em Step 2.} We now go back to  equation \eqref{eqdiff} and we multiply it by $v_n$, obtaining as $n\to+\infty$
		\[
		1 =\|v_n\|_{\lambda_n}^2= -\langle k_n, v_n\rangle_{\lambda_n} +  \langle h_n, v_n\rangle_{\lambda_n} +\int_\G f'(W_{\lambda_n} )v_n^2 \,dx=
		\int_\G f'(W_{\lambda_n} )v_n^2\,dx +o(1),
		\]
		since $ \langle k_n, v_n\rangle_{\lambda_n} = 0$ because $k_n\in K_{\lambda_n}$, $v_n\in K_{\lambda_n}^\perp$, and  $\langle h_n, v_n\rangle_{\lambda_n} \to 0$ by \eqref{assurdo}.
		In view of this, if we prove that
		\begin{equation}
		\label{intto0}
		\int_\G f'(W_{\lambda_n} )v_n^2\,dx \to0
		\end{equation}
		a contradiction arises and the proof is completed. To this aim, we  set 
		\[
		\tilde v_n(x)=\lambda_n^{1/4} v_n\big(x/\sqrt{\lambda_n}\big),\quad \tilde h_n(x)=\lambda_n^{1/4}h_n\big(x/\sqrt{\lambda_n}\big),\quad \tilde k_n(x)=\lambda_n^{1/4} k_n\big(x/\sqrt{\lambda_n}\big)\,.
		\]
		Note that $\tilde v_n, \tilde h_n, \tilde k_n$ are defined on the scaled graph $\G_n : =\sqrt{\lambda_n}\, \G$ and direct computations show that
		\begin{equation}
		\label{vtilde}
		\begin{split}
		\|\tilde v_n\|_{H^1(\G_n)}&\,=\|v_n\|_{\lambda_n}=1\\
		\int_{\G} f'(W_{\lambda_n} )v_n^2dx&\,=\int_{\G_n} \chi^{2\mu}\left(x/\sqrt{\lambda_n}\right)f'\left(\Psi(x)+  \sum_{i=1}^{N-1}b_{i,\lambda_n} Z^{(i)}(x) \right)\tilde v_n^2(x)\,dx
		\end{split}
		\end{equation}
		and 
		\[
		\|\tilde h_n\|_{H^1(\G_n)}=\|h_n\|_{\lambda_n}\to0,\quad \|\tilde k_n\|_{H^1(\G_n)}=\|k_n\|_{\lambda_n}\to0
		\]
		as $n\to+\infty$ by \eqref{assurdo}, \eqref{knto0}. Since, by construction, one can identify $\G_n\cap B(\overline{\vv},2\ell\sqrt{\lambda_n})$ with the compact subset of $\starN$ given by $\starN\cap B(0,2\ell\sqrt{\lambda_n})$, combining \eqref{eqdiff} with the previous formulas shows that, for every compactly supported $\varphi\in H^1(\starN)$, there exists $n$ large enough so that
		\[
		\begin{split}
		\int_{\starN}\tilde v_n'(x)\varphi'(x)+\tilde v_n(x)\varphi(x)-&\,f'\left(\chi\left(x/\sqrt{\lambda_n}\right)\left(\Psi(x)+  \sum_{i=1}^{N-1} b_{i,\lambda_n}Z^{(i)}(x)\right)\)\tilde v_n(x)\varphi(x)\,dx\\
		=&\,\int_{S_N}(\tilde h_n-\tilde k_n)'(x)\varphi'(x)+(\tilde h_n-\tilde k_n)(x)\varphi(x)\,dx=o(1)\,,
		\end{split}
		\]
		where with a slight abuse of notation we are thinking of $\tilde v_n, \tilde h_n, \tilde k_n$ as functions on $\starN\cap B(0,2\ell\sqrt{\lambda_n})$. Hence, $\tilde v_n$ converges weakly in $H^1$ and strongly in $L^q$, for every $q\geq2$, on compact subsets of $\starN$  to a function $v_0$. Arguing as in \cite[Lemma 3.1]{DGMP}, it is easy to see that $v_0 \in H^1(\starN)$ and, letting $n\to+\infty$ in the previous formula, that
		\[
		\int_{\starN}v_0'\varphi'+v_0\varphi-f'(\Psi)v_0\varphi=0, \quad\forall \varphi \in H^1(\starN),
		\]
		namely that $v_0 \in K$, where $K$ is as in Lemma \ref{non}. However, since arguing as above and recalling that $v_n\in K_{\lambda_n}^\perp$ one also has
		\[
		\langle v_0, Z^{(j)}\rangle_{H^1(\starN)}=\lim_{n\to+\infty}\lambda_n^{-\frac14-\frac1{2\mu}}\langle v_n,  \chi Z_{\lambda_n}^{(j)}\rangle_{\lambda_n}=0\qquad\forall j=1,\dots, N-1\,,
		\]
		i.e. $v_0\in K^\perp$, it follows that $v_0\equiv0$ on $\starN$. As a consequence, when $n\to+\infty$
		\[
		\int_{\G_n\cap B(\overline\vv,2\ell\sqrt{\lambda_n})} f'\left(\Psi+  \sum_{i=1}^{N-1}b_{i,\lambda_n}Z^{(i)}\right)\tilde v_n^2\,dx\to 
		\int_{\starN } f'(\Psi  ) v_0^2=0\,,
		\]
		which together with the second line of \eqref{vtilde} implies \eqref{intto0} and concludes the proof.
	\end{proof}
	
	\begin{proof}[Proof of Proposition \ref{resto}]
		We prove the claim with a suitable fixed point argument. The proof is divided in two steps.
		
		\smallskip
		{\em Step 1: estimates on $\EE$}. We begin by showing that as $\lambda\to+\infty$
		\begin{equation}
		\label{estE}
		\|\EE\|_{\lambda}\lesssim  \lambda^{\frac14+\frac1{2\mu}-2\alpha}.
		\end{equation}
		Indeed, by \eqref{e} and the definition of $W_{\lambda}$ and  $i_{\lambda}^*\left(f'(W_{\lambda})\right)$, it follows
		\begin{equation}
		\label{eqE}
		\begin{split}
		-\EE''+\lambda\EE=&\,\chi''\left(\Psi_{\lambda}+\sum_{i=1}^{N-1}b_{i,\lambda}Z_{\lambda}^{i}\right)+2\chi'\left(\Psi_{\lambda}+\sum_{i=1}^{N-1}b_{i,\lambda}Z_{\lambda}^{i}\right)'\\
		&\,+f(W_{\lambda})-\chi\left(f(\Psi_{\lambda})+f'(\Psi_{\lambda})\sum_{i=1}^{N-1}b_{i,\lambda}Z_{\lambda}^{(i)}\right)= E_1+E_2
		\end{split}
		\end{equation}
		where
		\begin{equation}
		\label{e1}
		E_1:=\chi''\bigg(\Psi_{\lambda}+\sum_{i=1}^{N-1}b_{i,\lambda}Z_{\lambda}^{(i)}\bigg)+2\chi'\bigg(\Psi_{\lambda}+\sum_{i=1}^{N-1}b_{i,\lambda}Z_{\lambda}^{(i)}\bigg)'+(\chi^{2\mu+1}-\chi)\bigg(\Psi_{\lambda}+\sum_{i=1}^{N-1}b_{i,\lambda}Z_{\lambda}^{(i)}\bigg)^{2\mu+1} 
		\end{equation}
		and
		\begin{equation}
		\label{e2}
		E_2:=\chi\left[f\(\Psi_{\lambda}+\sum_{i=1}^{N-1}b_{i,\lambda}Z_{\lambda}^{(i)}\)-f\(\Psi_{\lambda}\)-f'\(\Psi_{\lambda}\) \sum_{i=1}^{N-1}b_{i,\lambda} Z_{\lambda}^{(i)}\right]\,.
		\end{equation}
		Multiplying \eqref{eqE} by $\EE$, integrating over $\G$ and using the H\"older inequality we have
		\[
		\begin{split}
		\|\EE\|_{\lambda}^2&\,\lesssim \left(\|E_1\|_{L^2(\G)}+\|E_2\|_{L^2(\G)}\right)\|\EE\|_{L^2(\G)}\\
		&\,\leq \lambda^{-1/2}\left(\|E_1\|_{L^2(\G)}+\|E_2\|_{L^2(\G)}\right)\|\EE\|_{\lambda}
		\end{split}
		\]
		that is 
		\begin{equation}
		\label{E1+E2}
		\|\EE\|_{\lambda}\lesssim \lambda^{-1/2}\left(\|E_1\|_{L^2(\G)}+\|E_2\|_{L^2(\G)}\right)\,.
		\end{equation}
		Now, adapting the argument in the proof of \cite[Proposition 3.2]{DGMP} one easily sees that
		\begin{equation}
		\label{2E}
		\|E_1\|_{L^2(\G)}\lesssim \lambda^{-\tau}\qquad \forall\tau>0\,.
		\end{equation}
		As for $E_2$, by \eqref{in2} one has
		\[
		|E_2|\leq |\chi|\left[\Psi_{\lambda}^{2\mu-1}\left|\sum_{i=1}^{N-1}b_{i,\lambda}Z_{\lambda}^{(i)}\right|^2+\left|\sum_{i=1}^{N-1}b_{i,\lambda}Z_{\lambda}^{(i)}\right|^{2\mu+1}\right]\,,
		\]
		so that by \eqref{bj}
		\[
		\begin{split}
		\int_\G E_2^2\,dx&\,\lesssim \sum_{i=1}^{N-1}|b_{i,\lambda}|^4\int_\G|\Psi_{\lambda}|^{2(2\mu-1)}\left|Z_{\lambda}^{(i)}\right|^4\,dx+\sum_{i=1}^{N-1}b_{i,\lambda}^{2(2\mu+1)}\int_\G\left|Z_{\lambda}^{(i)}\right|^{2(2\mu+1)}\,dx\\
		&\,\lesssim\lambda^{\frac32+\frac1\mu}\sum_{i=1}^{N-1}|b_{i,\lambda}|^4\int_{\R^+}|\phi|^{2(2\mu-1)}|\phi'|^4\,dx+\lambda^{\frac32+\frac1\mu}\sum_{i=1}^{N-1}|b_{i,\lambda}|^{2(2\mu+1)}\int_{\R^+}|\phi'|^{2(2\mu+1)}\,dx\\
		&\,\lesssim \lambda^{\frac32+\frac1\mu-4\alpha}+\lambda^{\frac32+\frac1\mu-2(2\mu+1)\alpha}\lesssim\lambda^{\frac32+\frac1\mu-4\alpha}\qquad\text{as }\lambda\to+\infty\,,
		\end{split}
		\]
		since $(\mathfrak{b}_1,\dots,\mathfrak{b}_{N-1})\in C$ which is compact by assumption and $\mu\geq\frac12$. Coupling with \eqref{E1+E2} and \eqref{2E} yields \eqref{estE}.
		
		\smallskip
		{\em Step 2: the contraction mapping argument.} We consider the operator
		\[
		T:K_{\lambda}^\perp\to K_{\lambda}^\perp,\qquad T(v):=\left(\Pi_{\lambda}^\perp\LL\right)^{-1}\left(\Pi_{\lambda}^\perp\EE+\Pi_{\lambda}^\perp\NN(v)\right)
		\]
		which is well--defined by Lemma \ref{lem:invL}, and the set
		\[
		B_{\lambda}:=\left\{v\in K_{\lambda}^\perp\,:\,\|v\|_{\lambda}\leq c\lambda^{\frac14+\frac1{2\mu}-2\alpha}\right\}
		\]
		for a suitable constant $c>0$ to be chosen later. Note that, if we prove that $T$ has a unique fixed point $\Phi$ in $B_{\lambda}$, then $\Phi$ satisfies \eqref{rate} and
		\[
		\begin{split}
		\LL\Phi=\Pi_{\lambda}\LL(\Phi)+\Pi_{\lambda}^\perp\LL(\Phi)=&\,\Pi_{\lambda}\LL(\Phi)+\Pi_{\lambda}^\perp\EE+\Pi_{\lambda}^\perp\NN(\Phi)\\
		=&\,\EE+\NN(\Phi)+\Pi_{\lambda}\left(\LL(\Phi)-\NN(\Phi)-\EE\right)\,,
		\end{split}
		\]
		namely \eqref{lhe} with $c_i:=\langle\Pi_{\lambda}\left(\LL(\Phi)-\NN(\Phi)-\EE\right),\chi Z_\lambda^{i}\rangle_{\lambda}$, for every $i=1,\dots,N-1$, and the proof is completed.
		
		To this aim, we chose $c$ so that, for $\lambda$ large enough, $T$ is a contraction on $B_\lambda$. Observe first that, by Lemma \ref{lem:invL} there exists $c_1>0$ such that
		\begin{equation}
		\label{T}
		\begin{split}
		&\|T(v)\|_\lambda\leq c_1\|\Pi_{\lambda}^\perp\EE+\Pi_{\lambda}^\perp\NN(v)\|_\lambda\\
		&\|T(v_1)-T(v_2)\|_\lambda\leq c_1\|\Pi_{\lambda}^\perp\NN(v_1)-\Pi_\lambda^\perp\NN(v_2)\|_\lambda
		\end{split}
		\end{equation}
		for every $v,v_1,v_2\in B_\lambda$. Recalling \eqref{np}, we have
		\[
		\|\NN(v_1)-\NN(v_2)\|_\lambda^2=\int_\G\left(f(W_\lambda+v_1)-f(W_\lambda+v_2)-f'(W_\lambda)(v_1-v_2)\right)\left(\NN(v_1)-\NN(v_2)\right)\,dx
		\]
		so that by H\"older inequality
		\begin{equation}
		\label{N1}
		\|\NN(v_1)-\NN(v_2)\|_\lambda\leq\lambda^{-1/2}\left\|f(W_\lambda+v_1)-f(W_\lambda+v_2)-f'(W_\lambda)(v_1-v_2)\right\|_{L^2(\G)}\,.
		\end{equation}
		Now, recalling that $f(s)=(s^+)^{2\mu+1}$ and $\mu\geq\frac12$, the mean value theorem guarantees the existence of functions $\theta,\tilde\theta:\G\to[0,1]$ such that
		\[
		\begin{split}
		\int_\G\big|f(W_\lambda+v_1)&\,-f(W_\lambda+v_2)-f'(W_\lambda)(v_1-v_2)\big|^2\,dx\\
		=&\,(2\mu+1)^2(2\mu)^2\int_\G\big|W_\lambda+\tilde\theta(\theta v_1+(1-\theta v_2))\big|^{2(2\mu-1)}\big|\theta v_1+(1-\theta)v_2\big|^2\big|v_1-v_2\big|^2\,dx\,.
		\end{split}
		\]
		Since by definition $\|W_\lambda\|_\infty\lesssim \lambda^{\frac1{2\mu}}$, whereas by the $L^\infty$--Gagliardo--Nirenberg inequality of $\G$ (see e.g. \cite[Section 2]{AST}) and $v_i\in B_\lambda$, $i=1,2$,
		\[
		\|v_i\|_{L^\infty(\G)}^2\lesssim\|v_i\|_{L^2(\G)}\|v_i'\|_{L^2(\G)}\leq \lambda^{-1/2}\|v_i\|_{\lambda}^2\leq c^2\lambda^{\frac1\mu-4\alpha}\quad i=1,2\,,
		\]
		plugging into the previous formula yields
		\[
		\int_\G\big|f(W_\lambda+v_1)\,-f(W_\lambda+v_2)-f'(W_\lambda)(v_1-v_2)\big|^2\,dx\lesssim c^2\lambda^{2-4\alpha}\|v_1-v_2\|_{L^2(\G)}^2\leq c^2\lambda^{1-4\alpha}\|v_1-v_2\|_\lambda^2\,,
		\]
		which coupled with \eqref{N1} gives
		\[
		\|\NN(v_1)-\NN(v_2)\|_\lambda\lesssim c\lambda^{-2\alpha}\|v_1-v_2\|_\lambda\,.
		\]
		Combining with the second line of \eqref{T} shows that there exists a constant $c_2>0$ such that for every $\lambda$ large enough it holds
		\[
		\|T(v_1)-T(v_2)\|_\lambda\leq c_2 c \lambda^{-2\alpha}\|v_1-v_2\|_\lambda\qquad\forall v_1,v_2\in B_\lambda,
		\]
		i.e. $T$ is a contraction on $B_\lambda$ since $\alpha>0$ by assumption. Furthermore, by \eqref{T}, \eqref{estE} and the previous estimate with $v_1=v$, $v_2=0$ we obtain
		\[
		\|T(v)\|_\lambda\leq c_1\left(c_3\lambda^{\frac14+\frac1{2\mu}-2\alpha}+c_2c\lambda^{-2\alpha}\|v\|_\lambda\right)\leq c_1(c_3+c_2c^2\lambda^{-2\alpha})\lambda^{\frac14+\frac1{2\mu}-2\alpha}\,,
		\]
		for a suitable constant $c_3$. Hence, for sufficiently large $\lambda$ it is enough to choose e.g. $c=2c_1c_3$ to obtain that $T$ maps $B_\lambda$ into itself, thus concluding the proof. 
	\end{proof}

	\section{The finite-dimensional problem}
	\label{sec:bi}
	In view of Proposition \ref{resto}, to complete the proof of Theorem \ref{main} it is enough to solve the finite-dimensional problem in the coefficients $\mathfrak{b}_j$, $j=1,\dots, N-1$. This is done by  
	finding $\mathfrak b_1,...,\mathfrak b_{N-1}$ such that the numbers $c_i$'s in \eqref{lhe} are zero
	and so the function $W_\lambda+\Phi$ is a genuine solution of problem \eqref{pro3}.  In this section we start doing this by proving a result that relates the $\mathfrak{b}_j$'s we are looking for with the critical points of a suitable function on $\R^{N-1}$.
	To this end, we introduce $G:\R^{N-1}\to\R$ defined by
	\begin{equation}
	\label{gi}
	G(\mathfrak b_1,\dots,\mathfrak b_{N-1}):=\sum\limits_{j=1}^{N}\left(\sum_{k=1}^{N-1} \mathfrak b_k e^k_j\right)^3, 
	\end{equation}
	where the vectors $e^{\ell}$ are defined in \eqref{uj}. This function, which one usually refers to as the reduced energy, plays a crucial role in our discussion, as highlighted by the next result.

	\begin{proposition}
		\label{ridotto}
		Let $(\overline{\mathfrak b}_1,\dots,\overline{\mathfrak b}_{N-1})\in\R^{N-1}$ be an isolated critical point of $G$ with non-zero local degree. Then there exists $\lambda_0>0$ (depending on $(\overline{\mathfrak b}_1,\dots,\overline{\mathfrak b}_{N-1}))$ such that, for every $\lambda\ge\lambda_0$, there exists $(\mathfrak b_1^\lambda,\dots,\mathfrak b_{N-1}^\lambda)\in\R^{N-1},$ approaching $(\overline{\mathfrak b}_1,\dots,\overline{\mathfrak b}_{N-1})$ as $\lambda\to\infty$, such that
		the corresponding real numbers $c_1,\dots,c_{N-1}$ in \eqref{lhe} vanish. 
	\end{proposition}
	\begin{proof}
		For every $j=1,\dots, N-1$, multiplying \eqref{lhe} by $\chi Z_\lambda^{(j)}$ we have
		\begin{equation}\label{rido1}
		\langle\LL(\Phi)-\EE-\NN(\Phi), \chi Z_\lambda^{(j)} \rangle_\lambda=\sum\limits_{i=1}^{N-1}c_i\langle \chi Z^{(i)}_\lambda, \chi Z^{(j)}_\lambda \rangle_\lambda=c_j\left(a\lambda^{\frac1\mu+\frac12}+o\left(\lambda^{\frac1\mu+\frac12}\right)\right)
		\end{equation}
		for every $\lambda$ large enough and for a suitable constant $a>0$, since by the same computations in \eqref{Bn}, \eqref{est1}
		\[
		\langle \chi Z^{(i)}_\lambda,\chi Z^{(j)}_\lambda\rangle_\lambda=0\quad \hbox{if}\ i\not=j\qquad \hbox{and}\qquad  \left \| \chi Z^{(j)}_\lambda\right\|_\lambda^2=a \lambda^{\frac12+\frac{1}{\mu}}+o\left( \lambda^{\frac12+\frac{1}{\mu}}\right)
		\]
		as $\lambda\to+\infty$. Hence, system \eqref{rido1} is diagonal in the $c_j$'s and, to prove that it admits only the trivial solution $c_j=0$ for every $j$, it is enough to find suitable values of $\lambda$ and of the $\mathfrak{b}_j$'s for which the left hand side of \eqref{rido1} is equal to zero. 
		
		To this end, we show that, for sufficiently large $\lambda$, there exist $\mathfrak{b}_1^\lambda,\dots,\mathfrak{b}_{N-1}^\lambda$ as in the statement of the proposition and making the left hand side of \eqref{rido1} vanish by proving that
		\begin{align}\label{dome}
		\langle\mathscr L(\Phi)-\mathscr E-\mathscr N(\Phi)),\chi Z^{(j)}_\lambda \rangle_\lambda=-\lambda^{\frac12+\frac{1}{\mu}-2\alpha}A\frac{\partial G(\mathfrak b_1, \mathfrak b_2,...,\mathfrak b_{N-1}) }{\partial \mathfrak b_j}(1+o(1))\,,
		\end{align}
		where
		\[
		A:=\frac{\mu(2\mu+1)}3\int_{\R^+}\phi^{2\mu-1}\(\phi'\)^3dx\,.
		\]
		Observe that this is enough to conclude, since for large $\lambda$ we can interpret the right hand side of \eqref{dome} as $-A\lambda^{\frac12+\frac1\mu-2\alpha}$ times the $j$-th derivative of a small perturbation of $G$. Since by assumption $G$ admits a critical point $(\overline{\mathfrak b}_1,\dots,\overline{\mathfrak b}_{N-1})$ with non-zero local degree, such a perturbation admits a critical point too, say $(\mathfrak{b}_1^\lambda,\dots,\mathfrak{b}_{N-1}^\lambda)$, converging to $(\overline{\mathfrak b}_1,\dots,\overline{\mathfrak b}_{N-1})$ as $\lambda\to+\infty$.
		
		We are thus left to prove \eqref{dome}. First, we estimate the term 
		\begin{equation}
		\label{EZ}
		\langle \EE ,\chi Z^{(j)}_\lambda \rangle_\lambda=\int_{\G}  E_1 \chi Z^{(j)}_\lambda dx+\int_{\G}  E_2 \chi Z^{(j)}_\lambda dx\,,
		\end{equation}
		where $E_1,E_2$ as in \eqref{e1},\eqref{e2} (the previous identity follows by \eqref{eqE}). We have
		\[
		\begin{split}
		\left|\int_{\G} E_1 \chi Z^{(j)}_\lambda dx\right|&=\left|\int_{\G}\bigg[\chi''\bigg(\Psi_\lambda+\sum_{i=1}^{N-1}b_{i,\lambda}Z_\lambda^{(i)}\bigg)+2\chi'\bigg(\Psi_\lambda+\sum_{i=1}^{N-1}b_{i,\lambda}Z_\lambda^{(i)}\bigg)'\right.\\
		&\qquad\qquad\qquad\qquad\qquad\qquad\left.+(\chi^{2\mu+1}-\chi)\bigg(\Psi_\lambda+\sum_{i=1}^{N-1}b_{i,\lambda}Z_\lambda^{(i)}\bigg)^{2\mu+1}\bigg]\chi Z^{(j)}_\lambda dx\right|\\
		&\lesssim \lambda^{\frac1{\mu}}\int_{B(\overline{\vv},2\ell)\setminus B(\overline{\vv},\ell)}\left|\Psi(\sqrt{\lambda}\,x)+\sum_{i=1}^{N-1}b_{i,\lambda}Z^{(i)}(\sqrt{\lambda}\,x)\right|\left|Z^{(j)}(\sqrt{\lambda}\,x)\right|\,dx\\
		&\quad +\lambda^{\frac1\mu+\frac12}\int_{B(\overline{\vv},2\ell)\setminus B(\overline{\vv},\ell)}\left|\Psi'(\sqrt{\lambda}\,x)+\sum_{i=1}^{N-1}b_{i,\lambda}\big(Z^{(i)}\big)'(\sqrt{\lambda}\,x)\right|\left|Z^{(j)}(\sqrt{\lambda}\,x)\right|\,dx\\
		&\quad +\lambda^{\frac1\mu+1}\int_{B(\overline{\vv},2\ell)\setminus B(\overline{\vv},\ell)}\left|\Psi(\sqrt{\lambda}\,x)+\sum_{i=1}^{N-1}b_{i,\lambda}Z^{(i)}(\sqrt{\lambda}\,x)\right|^{2\mu+1}\left|Z^{(j)}(\sqrt{\lambda}\,x)\right|\,dx\,.
		\end{split}
		\]
		Since $b_{i,\lambda}=\lambda^{-\alpha}\mathfrak{b}_i$ with $\mathfrak{b}_i$'s in a bounded neighbourhood of $\left(\overline{\mathfrak b}_1,\dots,\overline{\mathfrak{b}}_{N-1}\right)$ and $Z^{(i)},\big(Z^{(i)}\big)'\in L^\infty(\starN)$, if $\lambda$ is large enough the previous estimate becomes
		\[
		\begin{split}
		\left|\int_{\G} E_1 Z^{(j)}_\lambda dx\right|&\,\lesssim\lambda^{\frac1\mu-\frac12}\int_{\ell\sqrt{\lambda}}^{2\ell\sqrt{\lambda}}|\phi||\phi|'\,dx+\lambda^{\frac1\mu}\int_{\ell\sqrt{\lambda}}^{2\ell\sqrt{\lambda}}|\phi'|^2\,dx+\lambda^{\frac1\mu+\frac12}\int_{\ell\sqrt{\lambda}}^{2\ell\sqrt{\lambda}}|\phi|^{2\mu+1}|\phi'|\,dx\,,
		\end{split}
		\]
		where $\phi$ is the soliton in \eqref{phi}, and since both $\phi,\phi'$ decay exponentially as $x\to+\infty$ this implies that 
		\begin{equation}
		\label{e1Z}
		\left|\int_\G E_1 Z_\lambda^{(j)}\,dx\right|\lesssim e^{-\sigma\sqrt{\lambda}} \qquad\text{as }\lambda\to+\infty
		\end{equation}
		for some $\sigma>0$. As for the second integral on the right hand side of \eqref{EZ}, by \eqref{e2} we write
		\begin{equation}
		\label{E2Z}
		\int_{\G}  E_2 \chi Z^{(j)}_\lambda \,dx=\frac12\int_{\G}\chi^2 f''(\Psi_\lambda) \(\sum_{i=1}^{N-1}b_{i,\lambda}Z^{(i)}_\lambda\)^2Z^{(j)}_\lambda dx+ \Theta(\lambda)
		\end{equation}
		with
		\[
		\begin{split}
		\Theta(\lambda):=&\int_{\G} \chi^2\left[f\(\Psi_\lambda+\sum_{i=1}^{N-1}b_{i,\lambda}Z_\lambda^{(i)}\)-f(\Psi_\lambda)-f'\(\Psi_\lambda\) \sum_{i=1}^{N-1}b_{i,\lambda} Z_\lambda^{(i)}\right.\\
		&\qquad\qquad\qquad\qquad\qquad\qquad\qquad\qquad\qquad\left.-\frac12 f''(\Psi_\lambda) \(\sum_{i=1}^{N-1}b_{i,\lambda}Z^{(i)}_\lambda\)^2\right]Z^{(j)}_\lambda dx\,.
		\end{split}
		\]
		Since $f(s)=(s^+)^{2\mu+1}$, by \eqref{in3} and the definition of $b_i$  we obtain
		\[
		\begin{split}
		\left|\Theta(\lambda)\right|&\,\lesssim\int_{\starN}\Psi_\lambda^{2\mu-2}\left|\sum_{i=1}^{N-1}b_{i,\lambda}Z_\lambda^{(i)}\right|^3\left|Z_\lambda^{(j)}\right|\,dx+\int_{\starN}\left|\sum_{i=1}^{N-1}b_{i,\lambda}Z_\lambda^{(i)}\right|^{2\mu+1}\left|Z_\lambda^{(j)}\right|\,dx\\
		&\,\lesssim \lambda^{1+\frac1\mu-3\alpha}\sum_{i=1}^{N-1}\int_{\starN}\Psi^{2\mu-2}(\sqrt{\lambda}\,x)\left|Z^{(i)}(\sqrt{\lambda}\,x)\right|^3\left|Z^{(j)(\sqrt{\lambda}\,x)}\right|\,dx\\
		&\quad+\lambda^{1+\frac1\mu-(2\mu+1)\alpha}\sum_{i=1}^{N-1}\int_{\starN}\left|Z^{(i)}(\sqrt{\lambda}\,x)\right|^{2\mu+1}\left|Z^{(j)}(\sqrt{\lambda}\,x)\right|\,dx\\
		&\,\lesssim\lambda^{\frac12+\frac1\mu-3\alpha}\int_{\R^+}\phi^{2\mu-2}\left|\phi'\right|^4\,dx+\lambda^{\frac12+\frac1\mu-(2\mu+1)\alpha}\int_{\R^+}|\phi'|^{2\mu+1}\,dx=o\left(\lambda^{\frac12+\frac1\mu-2\alpha}\right)\quad\text{as }\lambda\to+\infty\,,
		\end{split}
		\]
		where the finiteness of the integrals appearing in the last line is guaranteed by the properties of $\phi$ and the last equality makes use of $2\mu+1\geq2$. Combining with \eqref{E2Z} then yields
		\[
		\begin{split} 
		\int_{\G}  E_2 \chi Z^{(j)}_\lambda dx &=\frac12\int_{\G} \chi^2 f''(\Psi_\lambda) \(\sum_{i=1}^{N-1}b_{i,\lambda}Z^{(i)}_\lambda\)^2Z^{(j)}_\lambda dx+ o\left(\lambda^{\frac12+\frac1\mu-2\alpha}\right) \\ &=\mu(2\mu+1)\lambda^{\frac12+\frac{1}{\mu}} \int_{\starN}\Psi^{2\mu-1} \left(\sum_{i=1}^{N-1}b_{i,\lambda}Z^{(i)}\right)^2Z^{(j)} dx +  o\(\lambda^{\frac12+\frac{1}{\mu}-2\alpha}\)\\
		&=\mu(2\mu+1)\lambda^{\frac12+\frac{1}{\mu}-2\alpha} \int_{\R^+}\phi^{2\mu-1} (\phi')^3dx \underbrace{\sum_{k=1}^{N-1}\(\sum_{i=1}^{N-1}\mathfrak b_i e_k^{(i)}\)^2e_k^{(j)}}_{= \frac 13\frac{\partial G(\mathfrak b_1,\dots,\mathfrak b_{N-1}) }{\partial \mathfrak b_j}}+ o\(\lambda^{\frac12+\frac{1}{\mu}-2\alpha}\),
		\end{split}
		\]
		which coupled with \eqref{EZ}, \eqref{e1Z} entails
		\begin{equation}
		\label{eG}
		\langle \EE ,\chi Z^{(j)}_\lambda \rangle_\lambda=-\lambda^{\frac12+\frac1\mu-2\alpha}A\frac{\partial G(\mathfrak b_1, \dots,\mathfrak b_{N-1})}{\partial\mathfrak b_j}(1+o(1))
		\end{equation}
		for every $\lambda$ large enough.
		
		As for the term $\langle\NN(\Phi),\chi Z_\lambda^{(j)}\rangle_\lambda$, by \eqref{np}, \eqref{in2}, the fact that $b_{i,\lambda}=\lambda^{-\alpha}\mathfrak{b}_i\to0$ as $\lambda\to+\infty$, and the standard $L^\infty$ Gagliardo--Nirenberg inequality on $\G$ \cite[Section 2]{AST} we have
		\begin{equation}
		\label{NZ}
		\begin{split}
		\left|\langle \NN(\Phi) ,\chi Z^{(j)}_\lambda\rangle_\lambda\right| & =\left|\int_{\G}\left(f(W_\lambda+\Phi)-f(W_\lambda)-f'(W_\lambda)\Phi\right)\chi Z^{(j)}_\lambda dx\right| \\
		&\lesssim \int_{\G} \chi\left(|W_\lambda|^{2\mu-1}|\Phi|^2+|\Phi|^{2\mu+1}\right) \left|Z^{(j)}_\lambda\right| dx \lesssim \left(\lambda+\lambda^{\frac1{2\mu}}\|\Phi\|_{L^\infty(\G)}^{2\mu-1}\right)\|\Phi\|_{L^2(\G)}^2\\
		&\lesssim \left(1+\lambda^{\frac1{2\mu}-1}(\lambda^{-\frac14}\|\Phi\|_\lambda)^{2\mu-1}\right)\|\Phi\|_\lambda^2\lesssim \lambda^{\frac12+\frac1\mu-4\alpha}=o\left(\frac12+\frac1\mu-2\alpha\right)
		\end{split}
		\end{equation}
		since $\alpha>0$, the last inequality making use of \eqref{rate}.
		
		Finally, since $\Phi\in K^\perp_\lambda$, by \eqref{lp} we have
		\begin{equation}
		\label{LZ}
		\begin{split}
		\langle \LL(\Phi) , \chi Z^{(j)}_\lambda \rangle_\lambda &=-\int_{\G}f'(W_\lambda)\Phi\chi Z^{(j)}_\lambda\,dx\\
		&=-\int_{\G}\(f'(W_\lambda)-f'(\chi \Psi_\lambda ) \)\Phi \chi Z^{(j)}_\lambda\,dx -\int_{\G}f'(\chi \Psi_\lambda)\Phi \chi Z^{(j)}_\lambda\,dx=o\(\lambda^{\frac12+\frac{1}{\mu}-2\alpha}\)\,.
		\end{split}
		\end{equation}
		Indeed, by \eqref{in1}, the definition of $b_{i,\lambda}$, \eqref{rate} and $\alpha>0$
		\begin{align*}
		&\int_{\G}\left|f'(W_\lambda)-f'(\chi \Psi_\lambda ) \right|\left|\Phi \chi Z^{(j)}_\lambda\right|\,dx\\
		&\qquad \lesssim
		\int_{\G}\chi\Psi_\lambda ^{2\mu-1}\left|\sum b_{i,\lambda}  Z^{(i)}_\lambda\right|\left|\Phi Z^{(j)}_\lambda\right|\,dx+\int_\G\chi\left|\sum b_{i,\lambda}  Z^{(i)}_\lambda\right|^{2\mu}\left|\Phi Z^{(j)}_\lambda\right|\,dx\\
		&\qquad\lesssim\left( \left\|\Psi_\lambda^{2\mu-1}\left(\sum_{i=1}^{N-1}b_{i,\lambda}Z_\lambda^{(i)}\right)Z_\lambda^{(j)}\right\|_{L^2(\starN)}+\left\|\left(\sum_{i=1}^{N-1}b_{i,\lambda}Z_\lambda^{(i)}\right)^{2\mu}Z_\lambda^{(j)}\right\|_{L^2(\starN)}\right)\|\Phi\|_{L^2(\G)}\\
		&\qquad \lesssim\left(\lambda^{\frac34+\frac1{2\mu}-\alpha}+\lambda^{\frac34+\frac1{2\mu}-2\mu\alpha}\right)\|\Phi\|_{L^2(\G)}\lesssim\left(\lambda^{\frac14+\frac1{2\mu}-\alpha}+\lambda^{\frac14+\frac1{2\mu}-2\mu\alpha}\right)\|\Phi\|_\lambda=o\(\lambda^{\frac12+\frac{1}{\mu}-2\alpha}\)\,.
		\end{align*}
		Moreover, since $\Phi\in K_\lambda^\perp$,
		\begin{align*}-\int_{\G}f'(\chi \Psi_\lambda)\Phi\chi Z^{(j)}_\lambda\,dx&=\underbrace{\int_\G \Phi' \(\chi Z^{(j)}_\lambda\)'+\lambda \Phi\chi Z^{(j)}\,dx}_{=0}-\int_{\G}f'(\chi \Psi_\lambda)\Phi\chi Z^{(j)}_\lambda\,dx\\
		&=\int_{\starN}   \(-\(Z^{(j)}_\lambda\)''+\lambda  Z_\lambda^{(j)} - f'( \Psi_\lambda)Z^{(j)}_\lambda\)\chi\Phi\,dx+\sum_{e\succ \overline{\vv}}\Phi(\overline{\vv})\left(\chi Z^{(j)}_{\lambda}\right)_e'(\overline{\vv})\\
		&\quad+\int_{\G}\left(f'(  \Psi_\lambda)-f'(\chi \Psi_\lambda)\right)\Phi \chi Z^{(j)}_\lambda\,dx-\int_\G\left(\chi'' Z_\lambda^{(j)}+2\chi' \left(Z_\lambda^{(j)}\right)'\right)\Phi\,dx\,.
		\end{align*}
		By Lemma \ref{non}, the definition of $Z_\lambda^{(j)}$ and the properties of $\chi$, one has
		\[
		\int_{\starN}   \(-\(Z^{(j)}_\lambda\)''+\lambda  Z_\lambda^{(j)} - f'( \Psi_\lambda)Z^{(j)}_\lambda\)\chi\Phi\,dx=0
		\]
		and
		\[
		\sum_{e\succ \overline{\vv}}\Phi(\overline{\vv})\left(\chi Z^{(j)}_{\lambda}\right)_e'(\overline{\vv})=\sum_{e\succ \overline{\vv}}\Phi(\overline{\vv})\left(Z^{(j)}_{\lambda}\right)_e'(\overline{\vv})=0\,,
		\]
		whereas arguing as above one easily see that as $\lambda\to+\infty$
		\[
		\left|\int_{\G}\left(f'(  \Psi_\lambda)-f'(\chi \Psi_\lambda)\right)\Phi \chi Z^{(j)}_\lambda\,dx-\int_\G\left(\chi'' Z_\lambda^{(j)}+2\chi' \left(Z_\lambda^{(j)}\right)'\right)\Phi\,dx\right|\lesssim e^{-\gamma\sqrt{\lambda}}
		\]
		for some $\gamma>0$ (note that both integrals on the left hand side are nonzero only on $B(\overline{\vv},2\ell)\setminus B(\overline{\vv},\ell)$). Coupling all together proves \eqref{LZ}, which combined with \eqref{eG} and \eqref{NZ} yields  \eqref{dome} and completes the proof. 
	\end{proof}

	\section{The reduced energy and the end of the proof of Theorem \ref{main1}}
	\label{sec:G}
	
	This section characterizes the critical points of the reduced energy $G$ introduced in \eqref{gi}. We observe that the whole analysis developed so far is insensitive of the degree $N$ of the vertex $\overline{\vv}$, which for the results of Sections \ref{sec:phi}--\ref{sec:bi} to hold needs to be just greater than or equal to 2. Conversely, the result of this section is the only point in our work where we need to impose $N$  odd, as the next lemma clearly shows. 
	
	\begin{lemma}
		\label{grado} 
		For every odd $N \ge 3$, the unique critical point of $G$ is $0$ and its local degree satisfies 
		\[
		\mathtt{deg} (\nabla G, 0) = (-1)^{\frac{N-1}2} \binom{\scriptstyle{N-1}}{\frac{N-1}2}.
		\]
	\end{lemma}
	
	\begin{proof} We first note, for future reference, that from the definition \eqref{uj} of the vectors $e^i$ we have
		\begin{equation}
		\label{prop1}
		\sum_{j=1}^{N-1} e_j^i = -e_N^i = \begin{cases} 0 & \text{ if } i \le N-2 \\ N-1 & \text{ if } i = N-1. \end{cases}
		\end{equation}
		In view of this, we can write $G$ as 
		\[
		G(\mathfrak b_1,\dots,\mathfrak b_{N-1}):=\sum\limits_{j=1}^{N-1}\left(\sum\limits_{i=1}^{N-1} e^i_j \mathfrak b_i\right)^3 + \left( \sum_{i=1}^{N-1} e_N^i \mathfrak b_i \right)^3 = \sum\limits_{j=1}^{N-1}\left(\sum\limits_{i=1}^{N-1} e^i_j \mathfrak b_i\right)^3 - \big((N-1)  \mathfrak b_{N-1}\big)^3.
		\]
		Now we take the matrix $A := (e^i_j)_{ji}$, with $j,i=1,\dots,N-1$  and we change variables by setting $x := A\mathfrak b$, namely
		\[
		x_j := \sum_{i = 1}^{N-1}  e^i_j \mathfrak b_i,\qquad j= 1,\dots, N-1.
		\]
		Notice (from the definition of the vectors $e^i$) that the matrix $A$ satisfies $\det A \ne 0$, and that, by \eqref{prop1},
		\[
		\sum_{j=1}^{N-1} x_j = \sum_{j=1}^{N-1} \sum_{i = 1}^{N-1}  e^i_j \mathfrak b_i = 
		\sum_{i = 1}^{N-1} \mathfrak b_i \left(\sum_{j=1}^{N-1}  e^i_j\right) = - \sum_{i = 1}^{N-1} e_N^i \mathfrak b_i = (N-1) \mathfrak b_{N-1}.
		\]
		Defining $\oG : \R^{N-1} \to \R$ by $\oG(x) = G(A^{-1}x)$, so that $\oG(x) = G(\mathfrak b)$ whenever $x = A\mathfrak b$, there results
		\[
		\oG(x_1,\dots,x_{N-1}) = \sum_{j=1}^{N-1} x_j^3 - \left( \sum_{j=1}^{N-1} x_j \right)^3,
		\]
		and since $A$ is invertile, the local degree of $0$ is the same for $\nabla G$ and $\nabla \oG$.
		
		To find the critical points of $\oG$ we observe that
		\begin{equation}
		\label{deriv}
		\frac{\partial \oG}{\partial x_k}(x_1,\dots,x_{N-1}) =0  \qquad\text{if and only if  } \quad x_k^2 =  \left( \sum_{j=1}^{N-1} x_j \right)^2 \qquad\forall k =1,\dots, N-1.
		\end{equation}
		Since the right hand sides do not depend on $k$, this entails that $x_i^2 = x_k^2$ for every $i, k$, so that all critical points have the form
		\[
		(x_1,\dots,x_{N-1}) = (\sigma_1 t, \dots, \sigma_{N-1}t), \qquad \text{with }\quad \sigma_j \in \{+1,-1\}\quad \text{for every}\quad j = 1,\dots, N-1,
		\]
		for some $t \in \R$.
		Now, if $n$ denotes the number of negative $\sigma_j$'s in $(\sigma_1 t, \dots, \sigma_{N-1}t)$, then by \eqref{deriv}
		\[
		t^2 =  \left( \sum_{j=1}^{N-1} x_j \right)^2 = \big( (N-1-n) t - nt)\big)^2 = (N-1-2n)^2 t^2,
		\]
		or
		\begin{equation}
		\label{punticritici}
		\big( (N-1-2n)^2 -1\big) t^2 =0.
		\end{equation}
		As $N$ is odd, the coefficient of $t^2$ never vanishes, so that $\nabla \oG(x) = 0$ if and only if $x=0$, and the same holds of course for $\nabla G$. However, $x=0$ is degenerate, since the Hessian matrix of $\oG$ at $0$ is the null matrix. To compute the local degree of $0$ we therefore perturb $\oG$ by defining,
		for fixed small $\eps >0$,
		\[
		\oG_\eps(x_1,\dots, x_{N-1}) = \oG(x_1,\dots, x_{N-1}) - \eps^2 \sum_{j=1}^{N-1} x_j.
		\]
		Now, as above, we see that if $x$ is a critical point for $\oG_\eps$, then
		\begin{equation}
		\label{grad}
		x_k^2 =  \left( \sum_{j=1}^{N-1} x_j \right)^2 + \eps^2 \qquad\forall k =1,\dots, N-1.
		\end{equation}
		Since, again, the right hand sides do not depend on $k$, we find once more that  if $x$ is critical for $\oG_\eps$, then
		\[
		(x_1,\dots,x_{N-1}) = (\sigma_1 t, \dots, \sigma_{N-1}t), \qquad \text{with }\quad \sigma_j \in \{+1,-1\}\quad \text{for every}\quad j = 1,\dots, N-1,
		\]
		for some $t \in \R$. Denoting again by $n$ the number of negative $\sigma_j$'s, we see from \eqref{grad} that 
		\[
		t^2 =  \left( \sum_{j=1}^{N-1} x_j \right)^2 +\eps^2 = \big( (N-1-n) t - nt)\big)^2 +\eps^2 = (N-1-2n)^2 t^2 +\eps^2,
		\]
		namely
		\[
		\big( (N-1-2n)^2 -1 \big)t^2 = -\eps^2.
		\]
		Now it is immediate to check that the coefficient of $t^2$ is negative if and only if $n = \frac{N-1}{2}$, in which case the preceding equation reduces to $t^2 = \eps^2$, yielding $t = \pm \eps$.
		
		Therefore the critical points of $\oG_\eps$ are of the form
		\[
		x =  (\sigma_1 \eps, \dots, \sigma_{N-1}\eps)
		\]
		where $ \frac{N-1}{2}$ of the $\sigma_j$'s are negative and $ \frac{N-1}{2}$ positive. For this reason,  there are exactly $\displaystyle\binom{N-1}{\frac{N-1}2}$ such critical points.
		
		Finally, to compute the degree, we notice that the entries $h_{ij}(x)$ of the Hessian matrix $H\oG_\eps(x)$ have the form
		\[
		h_{ii}(x) = 6x_i - 6\sum_{j=1}^{N-1} x_j, \qquad h_{ij}(x) = - 6\sum_{j=1}^{N-1} x_j \quad\text{ if } i \ne j.
		\]
		But at a critical point $x$ we have
		\[
		\sum_{j=1}^{N-1} x_j = \sum_{j=1}^{N-1} \sigma_j\eps = \frac{N-1}2 \eps -  \frac{N-1}2 \eps = 0,
		\]
		and therefore 
		\[
		H\oG_\eps(x) = {\rm diag} \{6\sigma_1 \eps,\dots, 6\sigma_{N-1}\eps\},
		\]
		from which we obtain
		\[
		\det H\oG_\eps(x) = (-1)^{\frac{N-1}2} (6\eps)^{N-1}
		\]
		for every critical point of $\oG_\eps$. Thus, if $\eps < 1$,
		\[
		\mathtt{deg} (\nabla G, 0) =  \mathtt{deg} (\nabla \oG, B_1, 0)  = \mathtt{deg} (\nabla \oG_\eps, B_1, 0) = \sum_{x \in \oG_\eps^{-1}(0)} {\rm sgn} \det H\oG_\eps(x) =  (-1)^{\frac{N-1}2} \binom{\scriptstyle N-1}{\frac{\scriptstyle{N-1}}2},
		\]
		as there are exactly $\displaystyle\binom{N-1}{\frac{N-1}2}$ critical points.
	\end{proof}
	
	\begin{remark}
		\label{rem:Npari}
		If $N$ is even, by \eqref{punticritici} one can easily characterize the critical points of $\G$. To this end, consider the sets
		\[
		\mathcal{I}:=\left\{I\subset\{1,\dots,N-1\}\,:\,\#I=\frac N2\quad\text{or}\quad\#I=\frac N2-1\right\}
		\]
		and, for every $I\in\mathcal I$,
		\[
		\Sigma_I:=\left\{\sigma=(\sigma_1,\dots,\sigma_{N-1})\in\left\{-1,1\right\}^{N-1}\,:\,\sigma_i=-1\quad\Longleftrightarrow\quad i\in I\right\}\,.
		\]
		Then the set of all critical points of $G$ is given by
		\[
		\bigcup_{I\in\mathcal I}\bigcup_{\sigma\in\Sigma_I}\left\{\sigma t\,:\,t\in\R\right\}\,,
		\]
		namely the union of $\displaystyle2\binom{N-1}{\frac N2}$ straight lines passing through the origin in $\R^{N-1}$. In particular,  $\G$ has no isolated critical point and Proposition \ref{ridotto} does not apply anymore.
	\end{remark}
	
	\begin{proof}[Proof of Theorem \ref{main1}]
		Its is a direct consequence of Propositions \ref{resto}--\ref{ridotto} and Lemma \ref{grado}.
	\end{proof}
	
	\section{Proof of Theorem \ref{main2}}
	\label{sec:multi}
	
	This section is devoted to our main result about multi peaked solutions of \eqref{nlse} concentrating at finitely many vertices with odd degree. Since the argument of the proof is analogous to that developed so far for Theorem \ref{main1}, here we limit to sketch it highlighting the main differences.
	
	Let $M\ge2$ be a fixed natural number and $\overline{\vv}_1,\dots,\overline{\vv}_M\in\V$ be distinct vertices of $\G$ such that $N_i:=\dg(\overline{\vv}_i)\ge3$ is odd for every $i$. As pointed out in the Introduction, with no loss of generality we identify each $\overline{\vv}_i$ with $0$ along every edge $e\succ\overline{\vv}_i$. Furthermore, we set $\ell_{\overline{\vv}_i}:=\min_{e\succ\overline{\vv}_i}|e|/4$.
	
	To prove Theorem \ref{main2}, for every $\lambda$ large enough we set 
	\begin{equation}
	\label{multiW}
	W_\lambda:=\sum_{i=1}^M \chi_i\left(\Psi_{N_i,\lambda}+\sum_{j=1}^{N_i-1}b_{ij,\lambda}Z_\lambda^{(j)}\right)
	\end{equation}
	where, for every $i=1,\dots,M$, we take
	\begin{itemize}
		\item $\chi_i:\G\to[0,1]$ to be a smooth cut-off function such that $\chi_i\equiv1$ on $B(\overline{\vv}_i,\ell_{\overline\vv_i})$ and $\chi_i\equiv0$ on $\G\setminus B(\overline{\vv}_i,2\ell_{\overline{\vv}_i})$;
		
		\smallskip
		
		\item $\displaystyle\Psi_{N_i,\lambda}:=\lambda^{\frac1{2\mu}}\Psi_{N_i}\big(\sqrt{\lambda}\,x\big)$ for every $x\in\mathcal{S}_{N_i}$, with $\Psi_{N_i}$ as in \eqref{fi};
		
		\smallskip
		
		\item $\displaystyle Z_\lambda^{(j)}(x):=\lambda^{\frac1{2\mu}}Z^{(j)}\big(\sqrt{\lambda}\,x\big)$ for every $x\in\mathcal{S}_{N_i}$ and $j=1,\dots, N_i-1$, with $Z^{(j)}$ as in Lemma \ref{non};
		
		\smallskip
		
		\item $b_{ij,\lambda}=\lambda^{-\alpha}\mathfrak{b}_{ij}$, with $\alpha>0$ and $\mathfrak{b}_{ij}\in\R$ to de determined, for every $j=1,\dots, N_i-1$. 
	\end{itemize}
	We then look for positive solutions of \eqref{nlse} in the form $u_\lambda=W_\lambda+\Phi$, with $W_\lambda$ as in \eqref{multiW} and $\Phi$ a smaller order term to be find. According to Section \ref{sec:prel}, to solve this problem we look for $\Phi\in H^1(\G)$ and $\mathfrak{b}_{ij}\in\R$, $j=1,\dots, N_i-1,i=1,\dots,M$, such that
	\[
	\LL(\Phi)=\EE+\NN(\Phi)
	\]
	with $\LL,\NN,\EE$ as in \eqref{lp}, \eqref{np}, \eqref{e} respectively. 
	
	To this end, we introduce the sets
	\[
	\begin{split}
	&K_{M,\lambda}:=\left\{v\in H^1(\G)\,:\,v=\sum_{i=1}^M\chi_i\sum_{j=1}^{N_i-1}c_{ij}Z_\lambda^{(j)}\quad\text{for some }c_{ij}\in\R,\,j=1,\dots, N_i-1,\,i=1,\dots,M\right\}\\
	&K_{M,\lambda}^\perp:=\left\{v\in H^1(\G)\,:\,\langle v,\chi_i Z_\lambda^{(j)}\rangle_\lambda=0,\quad\forall j=1,\dots, N_i-1,\,i=1,\dots,M\right\}
	\end{split}
	\]
	and the corresponding projections $\Pi_{M,\lambda}:H^1(\G)\to K_{M,\lambda}$, $\Pi_{M,\lambda}^\perp:H^1(\G)\to K_{M,\lambda}^\perp$. Analogously to Section \ref{sec:phi}, we have the following.
	\begin{proposition}
		\label{prop:Mphi}
		For every compact subset $C$ in $\mathbb R^{(N_1-1)\times\dots\times(N_M-1)}$, there exists $\lambda_0>0$ (depending on $C$) such that, for every  $\displaystyle(\mathfrak b_{ij})_{\substack{1\le j\le N_i-1\\ i=1,\dots,M}}\in C$ and  every  $\lambda>\lambda_0$,
		there exist unique $\Phi\in K_{M,\lambda}^\perp$ and coefficients $\displaystyle(c_{ij})_{\substack{1\le j\le N_i-1\\ i=1,\dots,M}}\in \R^{(N_1-1)\times\dots\times(N_M-1)}$ (depending on $\Phi$) for which it holds
		\begin{equation}
		\label{multiLEN}
		\LL(\Phi)=\EE+\NN(\Phi)+\sum_{i=1}^M\chi_i\sum_{j=1}^{N_i-1}c_{ij}  Z^{(j)}_\lambda\,.
		\end{equation}
		Moreover, 
		\begin{equation*}
		\|\Phi\|_\lambda\lesssim  \lambda^{\frac14+\frac1{2\mu}-2\alpha}\,.
		\end{equation*}
	\end{proposition}
	\begin{proof}
		The argument is the same as that in the proof of Proposition \ref{resto}. 
		
		We first need to show that, for sufficiently large $\lambda$, the operator $\Pi_{\lambda}^\perp\LL$ is invertible with continuous inverse on $K_{M,\lambda}^\perp$. Since when $W_\lambda$ is as in \eqref{multiW} the operator $i^*(f'(W_\lambda))$ is compact, this can be done with a simple adaptation of the proof of Lemma \ref{lem:invL}: we assume the existence of $\lambda_n\to+\infty$, $v_n\in K_{M,\lambda_n}^\perp$ such that $\|v_n\|_{\lambda_n}=1$ for every $n$ and $\|\Pi_{M,\lambda_n}^\perp\LL(v_n)\|_{\lambda_n}\to0$ as $n\to+\infty$ and we seek a contradiction. Step 1 of the proof of Lemma \ref{lem:invL} works exactly the same here, simply with $\chi_i Z_{\lambda_n}^{(j)}$ in place of $\chi Z_{\lambda_n}^{(j)}$. To recover Step 2 we proceed as follows. On the one hand, we have again
		\[
		\int_\G f'(W_{\lambda_n})v_n^2\,dx=1+o(1)\qquad\text{as }n\to+\infty\,.
		\]
		On the other hand, we define again $\tilde{v}_n:={\lambda_n}^{\frac14}v_n\big(x/\sqrt{\lambda_n}\big)$ on the scaled graph $\G_n:=\sqrt{\lambda_n}\G$. Then, for every $i=1\dots,M$, we denote by $\tilde{v}_{n,i}$ the restriction of $\tilde{v}_n$ to $B(\overline{\vv}_i,2\ell_{\overline{\vv}_i}\sqrt{\lambda_n})$ in $\G_n$. Arguing as in Step 2 of the proof of Lemma \ref{lem:invL} it is easy to show that $\tilde{v}_{i,n}$ converges weakly in $H^1$ and strongly in $L^q$ on compact subsets of $\mathcal{S}_{N_i}$ to 0 as $n\to+\infty$ for every $i$, thus yielding
		\[
		\int_\G f'(W_{\lambda_n})v_n^2\,dx=\sum_{i=1}^M\int_{B(\overline{\vv}_i,2\ell_{\overline{\vv}_i}\sqrt{\lambda_n})}\chi_i^{2\mu}\big(x/\sqrt{\lambda_n}\big)f'\left(\Psi_{N_i}(x)+\sum_{j=1}^{N-i-1}b_{ij,\lambda_n}Z^{(j)}(x)\right)\tilde{v}_{n,i}^2\,dx=o(1)
		\]
		and providing the desired contradiction. 
		
		To conclude, we then need to recover the estimates of the error term $\EE$ and the contraction mapping argument. As for $\EE$, the computations are the same as in Step 1 of the proof of Proposition \ref{resto} with the terms $E_1,E_2$ that now read
		\[
		\begin{split}
		E_1:=\sum_{i=1}^M\chi_i''\bigg(\Psi_{N_i,\lambda}+\sum_{j=1}^{N_i-1}b_{ij,\lambda}Z_{\lambda}^{(j)}\bigg)&\,+2\sum_{i=1}^M\chi_i'\bigg(\Psi_{N_i,\lambda}+\sum_{j=1}^{N_i-1}b_{ij,\lambda}Z_{\lambda}^{(j)}\bigg)'\\
		&\,+\sum_{i=1}^M(\chi_i^{2\mu+1}-\chi_i)\bigg(\Psi_{N_i,\lambda}+\sum_{j=1}^{N_i-1}b_{ij,\lambda}Z_{\lambda}^{(j)}\bigg)^{2\mu+1}
		\end{split}
		\]
		and
		\[
		E_2:=\sum_{i=1}^M\chi_i\left[f\(\Psi_{N_i,\lambda}+\sum_{j=1}^{N_i-1}b_{ij,\lambda}Z_{\lambda}^{(j)}\)-f\(\Psi_{N_i,\lambda}\)-f'\(\Psi_{N_i,\lambda}\) \sum_{j=1}^{N_i-1}b_{ij,\lambda} Z_{\lambda}^{(j)}\right]\,.
		\]
		As for the contraction mapping argument, it is enough to repeat verbatim Step 2 of the proof of Proposition \ref{resto}.
	\end{proof}
	Theorem \ref{main2} is then a direct consequence of Proposition \ref{prop:Mphi} and of the next one.
	\begin{proposition}
		There exists $\lambda_0>0$ such that, for every $\lambda\ge\lambda_0$, there exists $(\mathfrak b_{ij}^\lambda)_{\substack{1\le j\le N_i-1\\ i=1,\dots,M}}\in\R^{(N_1-1)\times\dots\times(N_M-1)}$, with $\mathfrak{b}_{ij}^\lambda\to0$ as $\lambda\to\infty$ for every $j=1,\dots,N_i-1$, $i=1,\dots,M$, such that
		the corresponding real numbers $(c_{ij})_{\substack{1\le j\le N_i-1\\ i=1,\dots,M}}$ in Proposition \ref{prop:Mphi} vanish.
	\end{proposition}
	\begin{proof}
		For every $i,j$, multiplying \eqref{multiLEN} by $\chi_i Z_\lambda^{(j)}$ we obtain
		\[
		\langle\LL(\Phi)-\EE-\NN(\Phi), \chi_i Z_\lambda^{(j)} \rangle_\lambda=c_{ij}\langle\chi_i Z_\lambda^{(j)},\chi_i Z_\lambda^{(j)}\rangle_\lambda\,,
		\]
		which with the same computations of the proof of Proposition \ref{ridotto} can be rewritten as
		\[
		c_{ij}\left(a\lambda^{\frac12+\frac1\mu}+o\left(\lambda^{\frac12+\frac1\mu}\right)\right)=-\lambda^{\frac12+\frac{1}{\mu}-2\alpha}A\frac{\partial G_i(\mathfrak b_{i1},\dots,\mathfrak b_{i,N_i-1}) }{\partial \mathfrak b_{ij}}(1+o(1))
		\]
		where, for every $i=1,\dots,M$,
		\[
		G_i(\mathfrak b_{i1},\dots,\mathfrak b_{i,N_i-1}):=\sum\limits_{j=1}^{N_i-1}\left(\sum_{k=1}^{N_i-1} \mathfrak b_{ik} e^k_j\right)^3\,.
		\]
		Since $N_i$ is odd for every $i$, by Lemma \ref{grado} the origin in $\R^{N_i-1}$ is a critical point of $G_i$ with non-zero local degree.  Hence, the small perturbation of $G_i$ appearing in the above equations for $c_{ij}$, $j=1,\dots, N_i-1$, admits a critical point $(\mathfrak{b}_{i1}^\lambda,\dots,\mathfrak{b}_{i,N_i-1}^\lambda)\to(0,\dots,0)$ as $\lambda\to+\infty$, and we conclude.
	\end{proof}
	
	\section{Proof of Corollary \ref{cor1}}
	\label{sec:cor}
	
	We end the paper with the proof of Corollary \ref{cor1}, which in fact requires no more than combining a few simple estimates. On the one hand, if $u_\lambda$ is a one-peaked solution as in Theorem \ref{main1}, then by \eqref{u1picco} direct computations yield formula \eqref{massa1} for $\|u_\lambda\|_{L^2(\G)}^2$ and
	\[
	J_\lambda(u_\lambda)=\lambda^{\frac1\mu+\frac12}\left(\frac N2 J_1(\phi)+o(1)\right),\qquad E(u_\lambda)=\lambda^{\frac1\mu+\frac12}\left(\frac N2 E(\phi)+o(1)\right),
	\]
	whereas if $u_\lambda$ is a multi-peaked solution of Theorem \ref{main2}, by \eqref{umulti} we have \eqref{massamulti} and
	\[
	J_\lambda(u_\lambda)=\lambda^{\frac1\mu+\frac12}\left(\sum_{i = 1}^{M}\frac {N_i}2 J_1(\phi)+o(1)\right),\qquad E(u_\lambda)=\lambda^{\frac1\mu+\frac12}\left(\sum_{i=1}^M\frac {N_i}2 E(\phi)+o(1)\right)\,.
	\]
	Here, as usual $\phi$ is the soliton \eqref{phi} and $J_1(\phi)$, $E(\phi)$ are its action and energy on $\R$, respectively. 
	
	Conversely, for every $\omega>0$ set $\phi_\omega(x):=\omega^{\frac1{2\mu}}\phi(\sqrt{\omega}\,x)$, so that $\phi_\omega\in\mathcal{N}_\omega(\R)$ is the unique solution of \eqref{nlse} on $\R$ with $\omega=\lambda$. Since, for every $\lambda$ large enough, there exists $v_\lambda\in \mathcal{N}_\lambda(\R)$ with compact support on an interval of length $\frac2\lambda$ and such that $J_\lambda(v_\lambda)-J_\lambda(\phi_\lambda)=o(1)$ as $\lambda\to+\infty$, thinking of $v_\lambda$ as a function supported on any given edge of $\G$ yields
	\[
	\inf_{u\in\mathcal{N}_\lambda(\G)}J_\lambda(u)\leq J_\lambda(v_\lambda)=J_\lambda(\phi_\lambda)+o(1)=\lambda^{\frac1\mu+\frac12}J_1(\phi)+o(1).
	\]
	In particular, since $J_1(\phi)>0$ and $N,N_i\ge3$ by assumption, comparing with the previous formula implies
	\[
	J_\lambda(u_\lambda)>\inf_{u\in\mathcal{N}_\lambda(\G)}J_\lambda(u)
	\]
	for every $u_\lambda$ as in Theorems \ref{main1}--\ref{main2}. Hence, $u_\lambda$ is not a ground state of the action in $\mathcal{N}_\lambda(\G)$.
	
	As for the energy, let us distinguish the cases $\mu<2$, $\mu=2$, $\mu>2$.
	
	If $\mu>2$ there is nothing to prove, since it is well-known that $\displaystyle\inf_{u\in H_\nu^1(\G)}E(u)=-\infty$ for every $\G\in\mathbf{G}$ and $\nu>0$ (see e.g. \cite{AST}).
	
	If $\mu=2$, the discussion in \cite{ASTcmp} shows that $\displaystyle\inf_{u\in H_\nu^1(\G)}E(u)=-\infty$ for every $\G\in\mathbf{G}$ and $\nu>\|\phi\|_{L^2(\R)}^2$. Since $N, N_i\ge3$ by assumption, \eqref{massa1}, \eqref{massamulti} yield $\|u_\lambda\|_{L^2(\G)}>\|\phi\|_{L^2(\R)}$ for every $u_\lambda$ in Theorems \ref{main1}--\ref{main2}, thus proving again that it cannot be a ground state of the energy in the mass-constrained space.
	
	If $\mu<2$, note that for every $\nu>0$ one has
	\[
	\|\phi_\omega\|_{L^2(\R)}^2=\nu\qquad\Longleftrightarrow\qquad \omega=\left(\frac\nu{\|\phi\|_{L^2(\R)}^2}\right)^{\frac{2\mu}{2-\mu}}\,.
	\] 
	Hence, considering as above suitable compactly supported functions, it follows that
	\[
	\inf_{u\in H_\nu^1(\G)}E(u)\leq E\left(\phi_{\big(\nu/\|\phi\|_{L^2(\R)}^2\big)^{\frac{2\mu}{2-\mu}}}\right)+o(1)=\left(\frac\nu{\|\phi\|_{L^2(\R)}^2}\right)^{\frac{2+\mu}{2-\mu}}E(\phi)+o(1)
	\]
	as $\nu\to+\infty$. In particular, 
	\[
	\begin{split}
	\nu=\lambda^{\frac1\mu-\frac12}\frac N2\|\phi\|_{L^2(\R)}^2\qquad&\Longrightarrow\qquad \inf_{u\in H_\nu^1(\G)}E(u)\leq \left(\frac N2\right)^{\frac{2+\mu}{2-\mu}}\lambda^{\frac1\mu+\frac12}E(\phi)+o(1)\\
	\nu=\lambda^{\frac1\mu-\frac12}\sum_{i=1}^M\frac {N_i}2\|\phi\|_{L^2(\R)}^2\qquad&\Longrightarrow\qquad\inf_{u\in H_\nu^1(\G)}E(u)\leq \left(\sum_{i=1}^M\frac {N_i}2\right)^{\frac{2+\mu}{2-\mu}}\lambda^{\frac1\mu+\frac12}E(\phi)+o(1)\,.
	\end{split}
	\]
	Since $E(\phi)<0$ and $N,N_i\geq3$, by \eqref{massa1},\eqref{massamulti}, comparing with the explicit formulas above for $E(u_\lambda)$ proves again that $u_\lambda$ is not a ground state of the energy among the functions with the same mass and concludes the proof of Corollary \ref{cor1}.


\begin{thebibliography}{}
		
		\bibitem{ABD}
		Adami R., Boni F., Dovetta S., {\em Competing nonlinearities in NLS equations as source of threshold phenomena on star graphs}, J. Funct. Anal. {\bf 283}(1) (2022),
		109483, 34pp. 
		
		\bibitem{ABR}
		Adami R., Boni F., Ruighi A., {\em Non-Kirchhoff vertices and nonlinear Schr\"odinger ground states on graphs}, Mathematics {\bf 8}(4) (2020), 617.
		
		\bibitem{acfn_aihp}
		Adami R., Cacciapuoti C., Finco D., Noja D., 
		{\em Constrained energy minimization and orbital stability for the NLS equation on a star graph}, Ann. Inst. H. Poincar\'e (C) An. Non Lin. {\bf 31}(6) (2014), 1289--1310.
		
		\bibitem{acfn_pl}
		Adami R., Cacciapuoti C., Finco D., Noja D., {\em Stationary states of NLS on star graphs}, EPL (Europhysics Letters) {\bf 100}(1) (2012), 10003.
		
		\bibitem{ADST}
		Adami R., Dovetta S., Serra E., Tilli P., {\em Dimensional crossover with a continuum of critical exponents for NLS on doubly periodic metric graphs}, Anal. PDE {\bf 12}(6) (2019), 1597--1612.
		
		\bibitem{ASTbound}
		Adami R., Serra E., Tilli P., 
		\textit{Multiple positive bound states for the subcritical NLS equation on metric graphs},  Calc. Var. PDE {\bf 58}(5) (2019), 16pp.
		
		\bibitem{ASTcmp}
		Adami R., Serra E., Tilli P., 
		{\em Negative energy ground states for the $L^2$--critical NLSE on metric graphs},
		Comm. Math. Phys. {\bf 352}(1) (2017), 387--406.
		
		\bibitem{AST1} 
		Adami R., Serra E., Tilli P.,
		{\em NLS ground states on graphs.}
		Calc. Var. PDE {\bf 54}(1) (2015), 743--761.
		
		\bibitem{AST}
		Adami R., Serra E., Tilli P., {\em Threshold phenomena and existence results for NLS ground states on metric graphs}, J. Funct. Anal. {\bf 271}(1) (2016), 201--223.
		
		\bibitem{ACT}
		Agostinho F., Correia S.,  Tavares H., 
		{\em Classification and stability of positive solutions to the NLS equation on the $\mathcal{T}$-metric graph}, arXiv:2306.13521 [math.AP] (2023).
		
		\bibitem{atom}
		Amico L. et al., {\em Roadmap on Atomtronics: state of the art and perspective}, AVS Quantum Sci. {\bf 3}(3) (2021), 039201.
		
		\bibitem{BKKM}
		Berkolaiko G., Kennedy J.B., Kurasov P., Mugnolo D., {\em Surgery principles for the spectral analysis of quantum graphs}, Trans. Amer. Math. Soc. {\bf 372} (2019), 5153--5197.
		
		\bibitem{BK}
		Berkolaiko G., Kuchment P.,
		\emph{Introduction to quantum graphs},
		Mathematical Surveys and Monographs 186, American Mathematical Society, Providence, RI, 2013.
		
		\bibitem{BMP}
		Berkolaiko G., Marzuola J.L., Pelinovsky D.E., \emph{Edge--localized states on quantum graphs in the limit of large mass}, Ann. Inst. H. Poincar\'e (C) An. Non Lin. {\bf 38}(5) (2021), 1295--1335.
		
		\bibitem{BDL20}
		Besse C., Duboscq R., Le Coz S., {\em Gradient flow approach to the calculation of ground states on nonlinear quantum graphs}, Ann. H. Lebesgue {\bf 5} (2022), 387--428.
		
		\bibitem{BDL21}
		Besse C., Duboscq R., Le Coz S., {\em Numerical simulations on nonlinear quantum graphs with the GraFiDi library}, SMAI J. Comp. Math. {\bf 8} (2022), 1--47.
		
		\bibitem{BC}
		Boni F., Carlone R., {\em NLSE on the half-line with point interactions}, Nonlinear Diff. Eq. Appl. {\bf 30} (2023), art. n. 51.
		
		\bibitem{BD1}
		Boni F., Dovetta S., \emph{Ground states for a doubly nonlinear Schr\"odinger equation in dimension one}, J. Math. Anal. Appl. {\bf496}(1) (2021), 124797.
		
		\bibitem{BD2}
		Boni F., Dovetta S., {\em Doubly nonlinear Schr\"odinger ground states on metric graphs}, Nonlinearity {\bf35} (2022), 3283--3323.
		
		\bibitem{BCT1}
		Borrelli W., Carlone R., Tentarelli L.,
		{\em Nonlinear Dirac equation on graphs with localized nonlinearities: bound states and nonrelativistic limit},
		SIAM J. Math. Anal. {\bf51}(2) (2019), 1046--1081.
		
		\bibitem{BCT2}
		Borrelli W., Carlone R., Tentarelli L., {\em On the nonlinear Dirac equation on noncompact metric graphs}, J. Diff. Eq. {\bf278} (2021), 326--357.
		
		\bibitem{BCJS}
		Borthwick J., Chang  X., Jeanjean L., Soave N., {\em Normalized solutions of $L^2$--supercritical NLS equations on noncompact metric graphs with localized nonlinearities}, Nonlinearity {\bf36} (2023), 3776--3795.
		
		\bibitem{CJS}
		Chang X., Jeanjean L., Soave N., {\em Normalized solutions of $L^2$--supercritical NLS equations on compact metric graphs}, Ann. Inst. H. Poincar\'e (C) An. Non Lin. (2022). https://doi.org/10.4171/AIHPC/88.
		
		\bibitem{DDGS} 
		De Coster C., Dovetta S., Galant D., Serra E.,
		\emph{On the notion of ground state for nonlinear Schr\"odinger equations on metric graphs,} Calc. Var. PDE (2023) 62:159.
		
		\bibitem{DDGST}
		De Coster C., Dovetta S., Galant D., Serra E., Troestler  C., {\em Constant sign and sign changing NLS ground states on noncompact metric graphs}, arXiv:2306.12121 [math.AP] (2023).
		
		\bibitem{DGMP}
		Dovetta S., Ghimenti M., Micheletti A.M., Pistoia A., {\em Peaked and low action solutions of NLS equations on graphs with terminal edges}, SIAM J. Math. Anal. {\bf 52}(3) (2020), 2874--2894. 
		
		\bibitem{DST20}
		Dovetta S., Serra E., Tilli P., {\em NLS ground states on metric trees: existence results and open questions}, J. London Math. Soc.  {\bf 102}(3) (2020), 1223--1240.
		
		\bibitem{DT}
		Dovetta S., Tentarelli L., {\em Symmetry breaking in two--dimensional square grids: persistence and failure of the dimensional crossover}. J. Math. Pures Appl. {\bf 160} (2022), 99--157.
		
		\bibitem{FMN}
		Fijav\u{z} M.K., Mugnolo D., Nicaise S., {\em Linear hyperbolic systems on networks: well--posedness and qualitative properties}, ESAIM:COCV {\bf 27} (2021), art. n. 7. 
		
		\bibitem{HKMP}
		Hofmann M., Kennedy J.B., Mugnolo D., Pl\"umer M., {\em On Pleijel's nodal domain theorem for quantum graphs}, Ann. Henri Poincar\'e {\bf 22} (2021), 3841--3870.
		
		\bibitem{KMPX}
		Kairzhan A., Marangell R., Pelinovsky D.E., Xiao K., {\em Existence of standing waves on a flower graph}, J. Diff. Eq. {\bf 271} (2021), 719--763. 
		
		\bibitem{KNP}
		Kairzhan A., Noja D., Pelinovsky D.E., {\em Standing waves on quantum graphs}, J. Phys. A {\bf 55}(24) (2022), 243001.
		
		\bibitem{KP}
		Kairzhan A., Pelinovsky D.E., {\em Multi-pulse edge-localized states on quantum graphs}, Anal. Math. Phys. {\bf 11} (2021), art. n. 171. 
		
		\bibitem{KKLM}
		Kennedy J.B., Kurasov P., L\'ena C., Mugnolo D., {\em A theory of spectral partitions of metric graphs}, Calc. Var. PDE {\bf 60} (2021), art. n. 61.
		
		\bibitem{NPS}
		Noja D., Pelinovsky D.E., Shaikhova G., {\em Bifurcations and stability of standing waves in the nonlinear Schr\"odinger equation on the tadpole graph}, Nonlinearity {\bf 28}(7) (2015), 2343--2378.
		
		\bibitem{NP}
		Noja D., Pelinovsky D.E.,
		\emph{Standing waves of the quintic NLS equation on the tadpole graph}, Calc. Var. PDE {\bf 59}(5) (2020), art. n. 173.
		
		\bibitem{NRS}
		Noja D., Rolando S., Secchi S., {\em Standing waves for the NLS on the double-bridge graph and a rational--irrational dichotomy}, J. Diff. Eq. {\bf 266}(1) (2019), 147--178.
		
		\bibitem{pankov}
		Pankov A., {\em Nonlinear Schr\"odinger equations on periodic metric graphs}, Discrete Cont. Dyn. Syst. {\bf 38}(2) (2018), 697--714.
		
		\bibitem{PPVV}
		Pellacci B., Pistoia A., Vaira G.,  Verzini G.,
		\emph{Normalized concentrating solutions to nonlinear elliptic problems}, J. Diff. Eq. {\bf 275} (2021), 882--919.
		
		\bibitem{PS20}
		Pierotti D., Soave N., 
		{\em Ground states for the NLS equation with combined nonlinearities on non-compact metric graphs}, SIAM J. Math. Anal. {\bf 54}(1) (2022), 768--790.
		
		\bibitem{PSV}
		Pierotti D., Soave N., Verzini G.,
		\emph{Local minimizers in absence of ground states for the critical NLS energy on metric graphs}, Proc. Royal Soc. Edinb. Section A: Math. {\bf 151}(2) (2021), 705--733.
		
	\end{thebibliography}
\end{document}